\documentclass[12pt]{article}
\usepackage{t1enc}
\usepackage[latin1]{inputenc}
\usepackage[english]{babel}
\usepackage{amsmath,amsthm}
\usepackage{amsfonts}
\usepackage{amssymb}
\usepackage{latexsym}
\usepackage{graphicx}
\usepackage{epstopdf}
\usepackage{float}
\usepackage[natural]{xcolor}
\usepackage{float,color,fancybox,shapepar,setspace,hyperref}
\theoremstyle{plain}
\newtheorem{thm}{Theorem}[section]

\newtheorem{theorem}{Theorem}
\newtheorem{lemma}[thm]{Lemma}
\newtheorem{claim}[thm]{Claim}
\newtheorem{corollary}[thm]{Corollary}
\newtheorem{question}{Question}
\newtheorem*{conjecture*}{Conjecture}

\numberwithin{definition}{section}

\numberwithin{remark}{section}

\title{\LARGE Flexibility of planar graphs without $C_4$ and $C_5$\thanks{This work is supported by NSFC(11971270, 11631014, 11271006) of China and Shandong Province Natural Science Foundation (ZR2018MA001) of China }}
\author{Donglei Yang, Fan Yang\thanks{Corresponding author. E-mail address: yangfan5262@163.com.}\\
{\small School of Mathematics, Shandong University, Jinan 250100, China}
 }
\date{}

\begin{document}
\maketitle

\baselineskip 0.65cm
\begin{abstract}
Let $G$ be a $\{C_4, C_5\}$-free planar graph with a list assignment $L$. Suppose a preferred color is given for some of the vertices. We prove that if all lists have size at least four, then there exists an $L$-coloring respecting at least a constant fraction of the preferences.

\baselineskip 0.6cm {\bf Key words:} Planar graph, List coloring, Flexibility.
\end{abstract}
\section{Introduction}
\indent

In what follows, all graphs considered are simple, finite, and undirected, and we follow \cite{Bon} for the terminologies and notation not defined here. A plane graph is a particular drawing of a planar graph in the Euclidean plane. Given a plane graph $G$, we denote its vertex set, edge set, face set and minimum degree by $V(G)$, $E(G)$, $F(G)$ and $\delta(G)$, respectively. The degree $d(v)$ is the number of edges incident with $v$. A vertex $v$ is called a $k$-vertex ($k^{+}$-vertex, or $k^{-}$-vertex) if $d(v)=k$ ($d(v)\geq k$, or $d(v)\leq k$, resp.). For any face $f\in F(G)$, the degree of $f$, denoted by $d(f)$, is the length of the shortest boundary walk of $f$, where each cut edge is counted twice. Analogously, a $k$-face ($k^+$-face, or $k^{-}$-face) is a face of degree $k$ (at least $k$, or at most $k$, resp.). We write $f=u_1u_2\ldots u_nu_1$ if $u_1, u_2, \ldots, u_n$ are the boundary vertices of $f$ in the clockwise order, and for integers $d_1, \ldots d_n$, we say that $f$ is a $(d_1, \ldots  d_n)$-face if $d(u_i)=d_i$ for all $i\in\{1,2, \ldots, n\}$. We say that $f$ is a $(d_1^{+}, \ldots  d_n)$-face if $d(v_1)\geq d_1$ and $d(v_i)=d_i$ for all $i\in\{2, \ldots, n\}$; and similarly for other combinations. Let $f_k(v)$, $n_k(v)$ and $n_k(f)$ denote the number of $k$-faces incident with the vertex $v$, the number of $k$-vertices adjacent to the vertex $v$, and the number of $k$-vertices incident with the face $f$, respectively. Moreover, we use $\delta(f)$ to refer to the minimum degree of vertices incident with $f$.

A \emph{list assignment} $L$ for a graph $G$ is a function that to each vertex $v\in V(G)$ assigns a set $L(v)$ of colors, and an \emph{$L$-coloring} is a proper coloring $\phi$ such that $\phi(v)\in L(v)$ for all $v\in V(G)$. If $G$ has an $L$-coloring,
%We say that $L$ is an assignment for the graph $G$ if it assigns a list $L(v)$ of possible colors to each vertex $v$ of $G$. If $G$ has a proper coloring $\phi$ such that $\phi(v)\in L(v)$ for each vertex $v$ of $G$,
then we say that $G$ is \emph{$L$-colorable}.
%A graph $G$ is $k$-choosable if $G$ is $L$-colorable from every assignment $L$ of lists of size at least $k$.
Initiated by Dvo\v{r}\'{a}k, Norin, Postle \cite{listre}, a \emph{request} for a graph $G$ with a list assignment $L$ is a function $r$ with $\mathrm{dom}(r)\subseteq V(G)$ such that $r(v)\in L(v)$ for all $v\in \mathrm{dom}(r)$. For $\varepsilon >0$, a request $r$ is \emph{$\varepsilon$-satisfiable} if there exists an $L$-coloring $\phi$ of $G$ such that $\phi(v)=r(v)$ for at least $\varepsilon |\mathrm{dom}(r)|$ vertices $v\in \mathrm{dom}(r)$. We say that a graph $G$ with the list assignment $L$ is \emph{$\varepsilon$-flexible} if every request is $\varepsilon$-satisfiable. Furthermore, we emphasize a stronger weighted form. A \emph{weighted request} is a function $w$ that to each pair $(v,c)$ with $v\in V(G)$ and $c\in L(v)$ assigns a nonnegative real number. Let $w(G,L)=\sum_{v\in V(G), c\in L(v)}w(v,c)$. For $\varepsilon >0$, we say that $w$ is \emph{$\varepsilon$-satisfiable} if there exists an $L$-coloring $\phi$ of $G$ such that
\begin{equation*}
\sum_{v\in V(G)}w(v, \phi(v))\geq \varepsilon w(G,L).
\end{equation*}
We say that $G$ with the list assignment $L$ is \emph{weighted $\varepsilon$-flexible} if every weighted request is $\varepsilon$-satisfiable.

It is worth pointing out that a request $r$ is $1$-satisfiable if and only if the precoloring given by $r$ can be extended to an $L$-coloring of $G$. It is easy to deduce that weighted $\varepsilon$-flexibility implies $\varepsilon$-flexibility when we set there is only at most one color is requested at each vertex of $G$ and all such colors have the same weight $1$.

Dvo\v{r}\'{a}k, Norin and Postle \cite{listre} proposed this topic and studied some natural questions in the context. They obtained some elementary results as follows, we say a list assignment $L$ is an $f$-assignment if $|L(v)|\geq f(v)$ for all $v\in V(H)$.
\begin{description}
\item[$\bullet$] There exists $\varepsilon >0$ such that every planar graph with a $6$-assignment is $\varepsilon$-flexible.
\item[$\bullet$] There exists $\varepsilon >0$ such that every planar graph of girth at least five with a $4$-assignment is $\varepsilon$-flexible.
\item[$\bullet$] For every integer $d\geq2$, there exists $\varepsilon >0$ such that every graph of maximum average degree at most $d$ and choosability at most $d-1$ with a $(d+2)$-assignment is weighted $\varepsilon$-flexible.
\end{description}

Especially, in \cite{DMM}, Dvo\v{r}\'{a}k et al. studied the $4$-weighted flexibility of triangle-free planar graphs, and the result coincides with the choosability of triangle-free planar graphs. Later on, an interesting question is raised by Masa\v{r}\'{i}k \cite{MT} in the following.
\begin{question}\label{Q1}
Does there exist $\varepsilon >0$ such that every $C_4$-free planar graph $G$ with a $4$-assignment is weighted $\varepsilon$-flexible?
\end{question}
If it is true, this would be optimal in terms of choosability \cite{Xu}. However, it might be difficult to obtain such a result since even the procedure for getting the corresponding result of triangle-free is very involved. So far, triangle-free planar graphs are the only known result for $\varepsilon$-weighted flexibility with $4$-assignment.

Based on the above results, we proceed a step towards Question \ref{Q1} by proving the following theorem.

\begin{theorem}\label{thm3}
There exists $\varepsilon >0$ such that every $\{C_4, C_5\}$-free planar graph with a $4$-assignment is weighted $\varepsilon$-flexible.
\end{theorem}

In $2007$, Voigt \cite{VM} proved that there exists $\{C_4, C_5\}$-free planar graphs which are not $3$-choosable. Hence, our result is the best possible up to the list size.

The rest of the paper is organized as follows. In order to prove Theorem \ref{thm3}, in Section $2$, we introduce some basic natation and essential tools used in list coloring settings. Using discharging method, we first find some necessary reducible configurations in Section $3$, and then present our discharging rules and final analysis in Section $4$.

\section{Preliminaries}
We use $P_k$ to denote a path of order $k$. Let $H$ be a graph, $S\subseteq V(H)$, $S$ is called a \emph{$(P_3+P_4$)-independent} set if  $H$ contains neither $P_3$ nor $P_4$ connecting two vertices in $S$. Let \emph{$1_{S}$} denote the characteristic function of $S$, i.e., $1_{S}(v)=1$ if $v\in S$ and $1_{S}(v)=0$ otherwise. For functions that assign integers to vertices of $H$, we define addition and subtraction in the natural way, adding$/$subtracting their values at each vertex independently. For a function $f: V(H)\rightarrow \mathbb{Z}$ and a vertex $v\in V(H)$, let $f\downarrow v$ denote the function such that $(f\downarrow v)(w)=f(w)$ for $w\neq v$ and $(f\downarrow v)v=1$.

Let $G$ be a graph and $H$ be an induced subgraph of another graph $G$. For an integer $k\geq3$, let $\delta_{G,k}: V(H)\rightarrow \mathbb{Z}$ be defined by $\delta_{G,k}(v)=k-\deg_{G}(v)$ for each $v\in V(H)$. Then $H$ is said to be a \emph{$(P_3+P_4, k)$-reducible} induced subgraph of $G$ if
\begin{description}
\item[\textbf{(FIX)}] for every $v\in V(H)$, $H$ is $L$-colorable for every $((\deg_{H}+\delta_{G,k})\downarrow v)$-assignment $L$, and
\item[\textbf{(FORB)}] for every $(P_3+P_4, k)$-independent set $S$ in $H$ of size at most $k-2$, $H$ is $L$-colorable for every $(\deg_{H}+\delta_{G,k}-1_{S})$-assignment $L$.
\end{description}

Note that (\textbf{FORB}) implies that $\deg_{H}(v)+\delta_{G,k}(v)\geq2$ for all $v\in V(H)$.
%Before we proceed, let us give an intuition behind these definitions. Consider any assignment $L_0$ of lists of size $k$ to vertices of $G$. The function $\delta_{G,k}$ describes how many more (or few) available colors each vertex has compared to its degree. Suppose we $L_0$ color $G-V(H)$, and let $L'$ be the list assignment for $H$ obtained from $L_0$ by removing from the list of each vertex the colors of its neighbors in $V(G)\setminus V(H)$. In $L'$, each vertex $v\in V(H)$ has at least $\deg_{H}(v)+\delta_{G,k}(v)$ available colors, since each color in $L_0(v)\setminus L'(v)$ correspond to a neighbor of $v$ in $V(G)\setminus V(H)$. Hence, (FIX) requires that $H$ is $L'$-colorable even if we prescribe the color of any single vertex of $H$, and (FORB) requires that $H$ is $L'$-colorable even if we forbid to use one of the colors on the $(P_3+P_4)$-independent set $S$.

To prove weighted $\varepsilon$-flexibility, we use the following observation made by Dvo\v{r}\'{a}k et al.\cite{listre}
\begin{lemma}[\cite{listre}]\label{weighted}
Let $G$ be a graph and let $L$ be a list assignment for $G$. Suppose $G$ is $L$-colorable and there exists a probability distribution on $L$-colorings $\varphi$ of $G$ such that for every $v\in V(G)$ and $c\in L(v)$, $\mathrm{Pr}[\varphi(v)=c] \geq \varepsilon$. Then $G$ with $L$ is weighted $\varepsilon$-flexible.
\end{lemma}
Moreover, we use the following well-known result due to Thomassen \cite{Thomasdeg}.
\begin{lemma}[\cite{Thomasdeg}]\label{degree}
Let $G$ be a connected graph and $L$ a list assignment such that $|L(u)| \geq \deg(u)$ for all $u\in V(G)$. If either there exists a vertex $u\in V(G)$ such that $|L(u)|> \deg(u)$, or some $2$-connected component of $G$ is neither complete nor an odd cycle, then $G$ is $L$-colorable.
\end{lemma}
The following lemma provides an essential technique to deal with the weighted flexibility of graphs.
\begin{lemma}\label{C4P4}
For all integers $k\geq3$, $b\geq1$, there exists a positive constant $\varepsilon$ such that the following holds. Let $G$ be a $\{C_4, C_5\}$-free graph. If for every $Z\subseteq V(G)$, the graph $G[Z]$ contains an induced $(P_3+P_4)$-reducible subgraph with at most $b$ vertices, then $G$ with any assignment of list of size $k$ is weighted $\varepsilon$-flexible.
\end{lemma}
\begin{proof}
Before proving the Lemma, we give the following claim at first.
\begin{claim}\label{claim}
For every integer $k\geq3$, there exist $\varepsilon$, $\delta>0$ as follows. Let $G$ be a $\{C_4, C_5\}$-free graph and $L$ be an assignment of lists of size $k$ to vertices of $G$. Then there exists a probability distribution on $L$-colorings $\phi$ of $G$ such that
\begin{description}
\item[(i)] for all $v\in V(G)$ and $c\in L(v)$, we have $\mathrm{Pr}[\phi(v)=c]\geq \varepsilon$, and
\item[(ii)] for any $P_3+P_4$-independent subset $S$ with $|S|\leq k-2$, $\mathrm{Pr}[\phi(v)\neq c \, \text{for all $v\in S$}]\geq \delta^{|S|}$.
\end{description}
\end{claim}
\begin{proof}%[{\textbf{Proof of Claim \ref{claim}}}]
Let $\delta=(\frac{1}{k})^{b}$, $\varepsilon=(\frac{1}{k})^{b+k-2}$. Assume $G$ is a $\{C_4, C_5\}$-free graph, $Z\subseteq V(G)$, there exists $Y\subseteq Z$ of size at most $b$ such that $G[Y]\subseteq G[Z]$ and $G[Y]$ is $(P_3+P_4, k)$-reducible. We prove the claim by induction on $|V(G)|$. A random $L$-coloring $\phi$ of $G$ is chosen as follows: we choose an $L$-coloring $\phi_1$ of $G-Y$ at random from the probability distribution obtained by the induction hypothesis. Let $L'$ be the list assignment for $G[Y]$ defined by
\begin{equation*}
L'(v)=L(v)\backslash\{\phi_1(u):uv\in E(G), u\notin Y\}
\end{equation*}
for all $v\in Y$. Note that $|L'(v)|\geq \deg_{G[Y]}(v)+\delta_{G,k}(v)$ for all $v\in Y$, and thus $G[Y]$ has an $L'$-coloring by (\textbf{FORB}) applied with $S=\emptyset$. We choose an $L'$-coloring $\phi_0$ uniformly at random among all $L'$-colorings of $G[Y]$, and let $\phi$ be the union of the colorings $\phi_1$ and $\phi_0$. Note that $|L'(v)|\geq \deg_{G[Y]}(v)+\delta_{G,k}(v)$ for all $v\in Y$.

Firstly, we prove that (ii) holds. Let $S$ be a $(P_3+P_4)$-independent subset in $G$, $S=S_1\cup S_2$, where $S_1=S\cap (G-Y)$, $S_2=S\cap Y$. Obviously, both $S_1$ and $S_2$ are $P_3+P_4$-independent subsets in $G-Y$ and $Y$, respectively. By the induction hypothesis for $G-Y$, we obtain $\mathrm{Pr}[(\forall v\in S_1)\phi(v)\neq c]=\mathrm{Pr}[(\forall v\in S_1)\phi_1(v)\neq c]\geq \delta^{|S_1|}$. If $S_1=S$, then (ii) holds. Therefore, suppose that $|S_1|\leq |S|-1$. Fix $\phi_1$, then consider the probability that $\phi_0$ gives all vertices of $S_2$ colors different from c. For $v\in S_2$, let $L_c(v)=L'(v)\backslash \{c\}$, and for $v\in Y\backslash S_2$, let $L_c(v)=L'(v)$. Note that $|L_c(v)|\geq \deg_{G[Y]}(v)+\delta_{G,k}(v)-1_S(v)$ for all $v\in Y$. By (\textbf{FORB}), there exists an $L_c$-coloring of $G[Y]$ in that no vertex of $S_2$ is assigned color c. Since $\phi_0$ is chosen uniformly among the at most $k^{b}$ $L'$-colorings of $G[Y]$, we conclude that the probability that no vertex of $S_2$ is assigned color c by $\phi_0$ is at least $(\frac{1}{k})^{b}=\delta$. Consequently, under the assumption that $\phi_1$ does not assign color c to any vertex of $S_1$, the probability that $\phi_0$ does not assign color c to any vertex of $S_2$ is at least $\delta$. Hence,
\begin{align*}
&\mathrm{Pr}[\phi(v)\neq c,~\forall~v\in S]\\=&\mathrm{Pr}[\phi_1(u)\neq c,~\forall~u \in S_1; \phi_0(w)\neq c,~\forall~w\in S_2] \\
       \geq &(\frac{1}{k})^{b}\delta^{|S_1|}\\ \geq &\delta^{|S_1|+1}\\ \geq &\delta^{|S|}.
\end{align*}
as required.

Next, we prove that (i) holds. Consider any vertex $v\in V(G)$ and a color $c\in L(v)$. If $v\in V(G-Y)$, then $\mathrm{Pr}[\phi(v)=c]=\mathrm{Pr}[\phi_1(v)=c]\geq \varepsilon$ by the induction hypothesis for $G-Y$. Therefore, we assume that $v\in Y$. Let $S$ be the set of neighbors of $v$ in $V(G)\backslash Y$. Since $G$ is $\{C_4, C_5\}$-free and all vertices in $S$ has a common neighbor, $S$ is $P_3+P_4$-independent in $G-Y$. Furthermore, (\textbf{FORB}) implies $1\leq \deg_{G[Y]}(v)+\delta_{G,k}(v)-1_{v}(v)=\deg_{G[Y]}(v)+k-\deg_G(v)-1=k-1-|S|$, thus $|S|\leq k-2$, By (ii) we obtain that no vertex of $S$ is assigned color c by $\phi_1$ with probability at least $\delta^{k-2}$. (\textbf{FIX}) implies that there exists an $L'$-coloring of $G[Y]$ in which $v$ is assigned color $c$. Since $\phi_0$ is chosen uniformly among all $L'$-colorings of $G[Y]$, it follows that under the assumption that no vertex of $S$ is colored by c, the probability that $\phi(v)=c$ is at least $\delta$. Therefore,
\begin{align*}
&\mathrm{Pr}[\phi(v)=c~\forall~v\in V(G)]\\=&\mathrm{Pr}[\phi_1(w)\neq c\ for~\forall~w \in S; \phi_0(u)=c~\forall~u\in Y] \\
\geq &(\frac{1}{k})^{b}\delta^{|S|}\\ \geq &(\frac{1}{k})^{b+k-2}.
\end{align*}
\medskip
This completes the proof of Claim \ref{claim}.
\end{proof}
Combining Claim \ref{claim} with Lemma \ref{weighted}, we obtain that Lemma \ref{C4P4} holds.
\end{proof}

Recall that a \emph{block} of $H$ is a maximal connected subgraph without a cutvertex. %, we call $G$ is $L\downarrow v$-\emph{colorable} if the coloring of $v$ can be extended to a list coloring of $G$.
The following lemma is easy to obtain by induction on the number of blocks in $H$. So we omit the proof.% by the definition of $L\downarrow v$-colorable.
\begin{lemma}\label{fix}
Let $L$ be an $f$-assignment for $H$. If every block $B$ of $H$ is $L'$-colorable for any $f\downarrow v$-assignment $L'$ with any fixed vertex $v\in V(B)$, then $H$ satisfies (\textbf{FIX}).
\end{lemma}

\section{Reducible subgraphs}
\medskip

Before proceeding, we introduce the following notation. A vertex $v$ with $4\leq d(v)\leq12$ is \emph{bad} if (i) the number of $(3,4^-,v)$-face is $\lfloor\frac{d(v)-2}{2}\rfloor$; (ii) $f_3(v)=\lfloor\frac{d(v)}{2}\rfloor$. A vertex $v$ with $d(v)=4$ is \emph{vice} if $f_3(v)=2$.
And a vertex $v$ with $5\leq d(v)\leq12$ is \emph{dangerous} if $f_{3,3}(v)=\lfloor\frac{d(v)-3}{2}\rfloor$. On the other hand, each $(3, 4, 4)$-face is called \emph{poor}, each $(3,4,v)$-face is called \emph{worse} and each $(3,3,v)$-face is called \emph{worst}. Furthermore, denote by $f_{3,3}(v)$ (or $f_{3,4}(v)$), $f_{3b}(v)$ (or $f_{4b}(v)$) the number of worst $3$-face (or worse $3$-face incident with $v$), the number of $(3,w,v)$-face (or $(4,w,v)$-face) in which $w$ is a bad vertex, respectively. In particular, let $f_{bb}(v)$ be the number of $(v, w_1, w_2)$-face in which both $w_1$ and $w_2$ are bad vertices. For each $6^+$-face $f$, denote by $\xi(f)$ the number of poor $3$-face sharing an edge with $f$. A \emph{nice} path connecting $v$ and some vertex $u\in V(f)$ is a path of length at most two such that either
\begin{description}
\item[(i)] $d(u)=3$ and all internal vertices have degree $3$ in $G'$; or
\item[(ii)] $d(u)=4$ and all internal vertices are vice $4$-vertices, moreover, every two consecutive $4$-vertices in the path are contained in a $(4,4,4)$-face.
\end{description}

Note that in all figures of the paper, any vertex marked with $\bullet$ has no edges of $G$ incident with it other than those shown, any vertex marked with $\blacksquare$ is bad unless stated otherwise.
%\begin{figure}[H]
%\begin{center}
%\includegraphics[scale=0.7]{4total00-eps-converted-to.pdf}\\
%{Figure 1. Reducible configurations}
%\end{center}
%\label{sub}
%\end{figure}

Here, we describe a quite general case of $(P_3+P_4, 4)$-reducible configurations. Let $G$ be a $\{C_4, C_5\}$-free planar graph and $v$ a vertex of $G$. A \emph{$v$-stalk} is one of the following subgraphs:
\begin{itemize}
\item[(a)] An edge $vu_1$;
\item[(b)] A path $vu_1u_2$;
\item[(c)] A cycle $vu_1u_2$;
\item[(d)] A cycle $vu_1v_1$;
\item[(e)] A path $vu_1u_2$ and a path $vv_1u_1$;
\item[(f)] A cycle $vu_1w_1$;
\item[(g)] A path $vu_1u_2$ and a path $vw_1u_1$;
\item[(h)] A cycle $vw_1w_2$;
\item[(i)] A cycle $vv_1w_1$;
\item[(j)] A path $vv_1v_2v_3$, a path $vw_1v_1$ and a path $v_2u_1v_3$ ($v_1$, $v_2$ may be the same vertex);
\item[(k)] A path $vv_1v_3v_4$, a path $vv_2v_1$ and a path $v_3u_1v_4$ ($v_1$, $v_3$ may be the same vertex);
\item[(l)] A path $vv_1v_3v_4$, a path $v_3u_1v_4$, a path $vv_2v_5v_6$, a path $v_5u_2v_6$ and an edge $v_1v_2$ ($v_1$, $v_3$ may be the same vertex, $v_2$, $v_5$ may be the same vertex);
\item[(m)] A path $vv_1v_2v_3$ and a path $v_2u_1v_3$ ($v_1$, $v_2$ may be the same vertex),

%\item[(1)] A path $vv_1$, i.e. $\mathrm{(a)}$;
%\item[(2)] A cycle $vv_1v_2$, i.e. $\mathrm{(b)}, \mathrm{(c)}, \mathrm{(d)}, \mathrm{(e)}, \mathrm{(f)}$;
%\item[(3)] A path $vv_1v_3v_4$, a path $v_3v_5v_4$, a path $vv_2v'_3v'_4$ and a path $v'_3v'_5v'_4$, i.e. $\mathrm{(g)}$;
%\item[(4)] A path $vv_1v_3v_4$, a path $v_3v_5v_4$ and a path $vv_2v_1$, i.e. $\mathrm{(h)}, \mathrm{(i)}$;
%\item[(5)] A path $vv_1v_3v_4$ and a path $v_3v_5v_4$, i.e. $\mathrm{(j)}$,
\end{itemize}
where $d(u_i)=3$, $d(v_i)=4$, $w_i$ is bad and $5\leq d(w_i)\leq12$ for each $i\in \{1, 2, \ldots, 6\}$. It is worth noting that for each bad vertex $w_i$, the $v$-stalk includes all neighbors of $w_i$ lying on the worse or worst $3$-faces.

\begin{lemma}\label{CP}
Assume that $G$ is a $\{C_4, C_5\}$-free planar graph and $v\in V(G)$ with $5\leq d(v) \leq12$. Let $A=\{\mathrm{(a)},\mathrm{(b)},\mathrm{(c)},\mathrm{(d)},\mathrm{(e)}, \mathrm{(f)}, \mathrm{(g)}, \mathrm{(h)}, \mathrm{(j)}, \mathrm{(l)} \}$, $B=\{\mathrm{(a)}, \mathrm{(b)}, \mathrm{(c), \mathrm{(d)}, \mathrm{(e)}, \mathrm{(l)}}\}$, $C=\{\mathrm{(b)}, \mathrm{(c)}, \mathrm{(e)}, \mathrm{(l)}\}$, $D=\{\mathrm{(a)}, \mathrm{(b)}, \mathrm{(c)}, \mathrm{(e)}, \mathrm{(f)}, \mathrm{(g)}, \mathrm{(h)}, \mathrm{(j)}, \mathrm{(l)}\}$ and $K=\{\mathrm{(i)}, \mathrm{(k)}\}$  be the sets of stalks of $v$. If there exists a induced subgraph $H$ containing $v$ such that one of the following holds:
\begin{description}
\item[(1)] $V(H)$ either consists of the vertices lying on any combination of elements in $A$, or any combination of elements in $B$ together with one copy of $\mathrm{(i)}$, or any combination of elements in $C$ together with one copy of $\mathrm{(k)}$, or any combination of elements in $D$ together with one copy of $\mathrm{(m)}$ such that the resulting $(\deg_{H}+\delta_{G,4})$-assignment $L$ satisfies $|L(v)|\geq3$.
\item[(2)] $V(H)$ consists of the vertices lying on any combination of elements in $A$ together with only one element in $K$, except the special case where $\mathrm{(d)}$ and $\mathrm{(k)}$ appear in the same combination, such that the resulting $(\deg_{H}+\delta_{G,4})$-assignment $L$ satisfies $|L(v)|\geq4$.
\end{description}
Then $G$ contains a $(P_3+P_4, 4)$-reducible induced subgraph with at most $138$ vertices.
\end{lemma}
\begin{proof}
%For convenience we repeat the relevant definition of $k$-sum from \cite{Bon}. Let $G_1$ and $G_2$ be two graphs whose intersection $G_1\cap G_2$ is a complete graph on $k$ vertices. The graph obtained from their union $G_1\cup G_2$ by deleting the edges of $G_1\cap G_2$ is called the $k$-sum of $G_1$ and $G_2$.
For $x\in N(v)$, let $T_x$ be a $v$-stalk witnessing this case. Let $H$ be the subgraph of $G$ induced by $\cup_{x\in N(v)}T_x$. Clearly, $|V(H)|\leq (2d(w)-1)\times \lfloor\frac{d(v)}{2}\rfloor \leq138$. Thus it suffices to show that $H$ is $(P_3+P_4, 4)$-reducible.

By $\{C_4, C_5\}$-freeness, it is easy to obtain that all stalks are pairwise vertex disjoint and there are no edges among them. Consider any vertex $z\in V(H)$ and a $(\deg_{H}+\delta_{G,4})\downarrow z$-assignment $L$ for $H$, $H$ is $L$-colorable by Lemma \ref{fix}, implying (\textbf{FIX}). Thus it only remains to show that $H$ satisfies (\textbf{FORB}). Let $S=\{s_1, s_2\}$ be a $(P_3+P_4)$-independent subset in $G$ and $L'$ a $(\deg_H+\delta-1_S)$-assignment for $H$. First, we performing the following operations:
\begin{description}
\item[$\bullet$] Arbitrarily choose any two $v$-stalks $T_1$ and $T_2$ which containing at least a bad vertex or a nice path randomly;
\item[$\bullet$] Arbitrarily choose a vertex $s_i$ in each $T_i$ for $i\in \{1,2\}$ with $dist(s_1, s_2)=1$ or $dist(s_1, s_2)\geq4$.
\end{description}

When $|L(v)|\geq3$, if $H$ is induced by the vertices of any combination of elements in $A$, let $A_0=\{\mathrm{(a)}, \mathrm{(b)}, \mathrm{(c)}, \mathrm{(d)}\}$. Obviously, the result holds for $A_0$. Next, we add $\mathrm{(e)}$, $\mathrm{(f)}$, $\mathrm{(g)}$, $\mathrm{(h)}$, $\mathrm{(j)}$, $\mathrm{(l)}$ into $A_0$ in turn to obtain $A$ and prove that the result also admits for $A$. Let $A_1=A_0 \cup \{\mathrm{(e)}\}$, we only need to consider whether (\textbf{FORB}) holds for the combination $\mathrm{(b)}$ and $\mathrm{(e)}$. In this situation, we can first $L'$-color $s_1, s_2, v$ in order and then greedily $L'$-color the remaining vertices. Hence, $A_1$ admits the result. Let $A_2=A_1 \cup \{\mathrm{(f)}\}$, we perform the above operation to define $S$, after that we $L'$-color the blocks which containing $s_1, s_2, v$ respectively, this is possible since $|L'|\geq3$, and then greedily $L'$-color the remaining vertices. By the same argument, we eventually find that $A$ admits the result.

If $H$ is induced by the vertices of any combination of elements in $B$ as well as only one copy of $\mathrm{(i)}$, then we first greedily $L'$-color the vertices around the bad vertex $w_1$ in $\mathrm{(i)}$ and then greedily $L'$-color the remaining vertices.

Similarly, if $H$ is induced by the vertices of any combination of elements in $C$ as well as only one copy of $\mathrm{(k)}$ or by the vertices of any combinations of elements in $D$ as well as only one copy of $\mathrm{(m)}$, it is easy to verify the result also holds.

When $|L(v)|\geq4$ and $H$ is induced by the vertices of any combination of elements in $A$ together with only one element in $K$. As usual, we perform the above operations. Then we first $L'$-color the vertices in the block which containing $s_i$, then greedily $L'$-color the remaining vertices. It is possible since $|L'(v)|\geq4$. Hence, in all the cases, (\textbf{FORB}) holds.

In conclusion, $H$ is $(P_3+P_4, 4)$-reducible.
\end{proof}
\begin{lemma}\label{4redu60}
Let $G$ be a $\{C_4,C_5\}$-free planar graph. If $G$ contains one of the following configurations (see Figure $1$). Here, $d(u_i)=3$, $d(v_i)=4$ and $5\leq d(w_i)\leq12$ for all $i$.
\begin{description}
\item[(1)] A $4$-vertex $v$ with $f_{bb}(v)=1$;
\item[(2)] A $6$-vertex $v$ with $f_{3,3}(v)=1, f_{3b}=1$;
\item[(3)] A $2t$-vertex $v$ with $f_{3,3}(v)=t-1$, where $2\leq t \leq6$;
\item[(4)] A $2t$-vertex $v$ with $f_{3,3}(v)=t-2, f_{bb}(v)=1$, where $2\leq t \leq6$;
\item[(5)] A $5$-vertex $v$ with $f_{3b}=1$, and a path $vu_2u_3$;
\item[(6)] A $6$-vertex $v$ with $f_{3b}=1$, and two paths $vu_2u_3$, $vu_4u_5$,
\end{description}
Then $G$ contains a $(P_3+P_4, 4)$-reducible induced subgraph with at most $29$ vertices.
\end{lemma}
\begin{figure}[H]
\begin{center}
\includegraphics[scale=0.8]{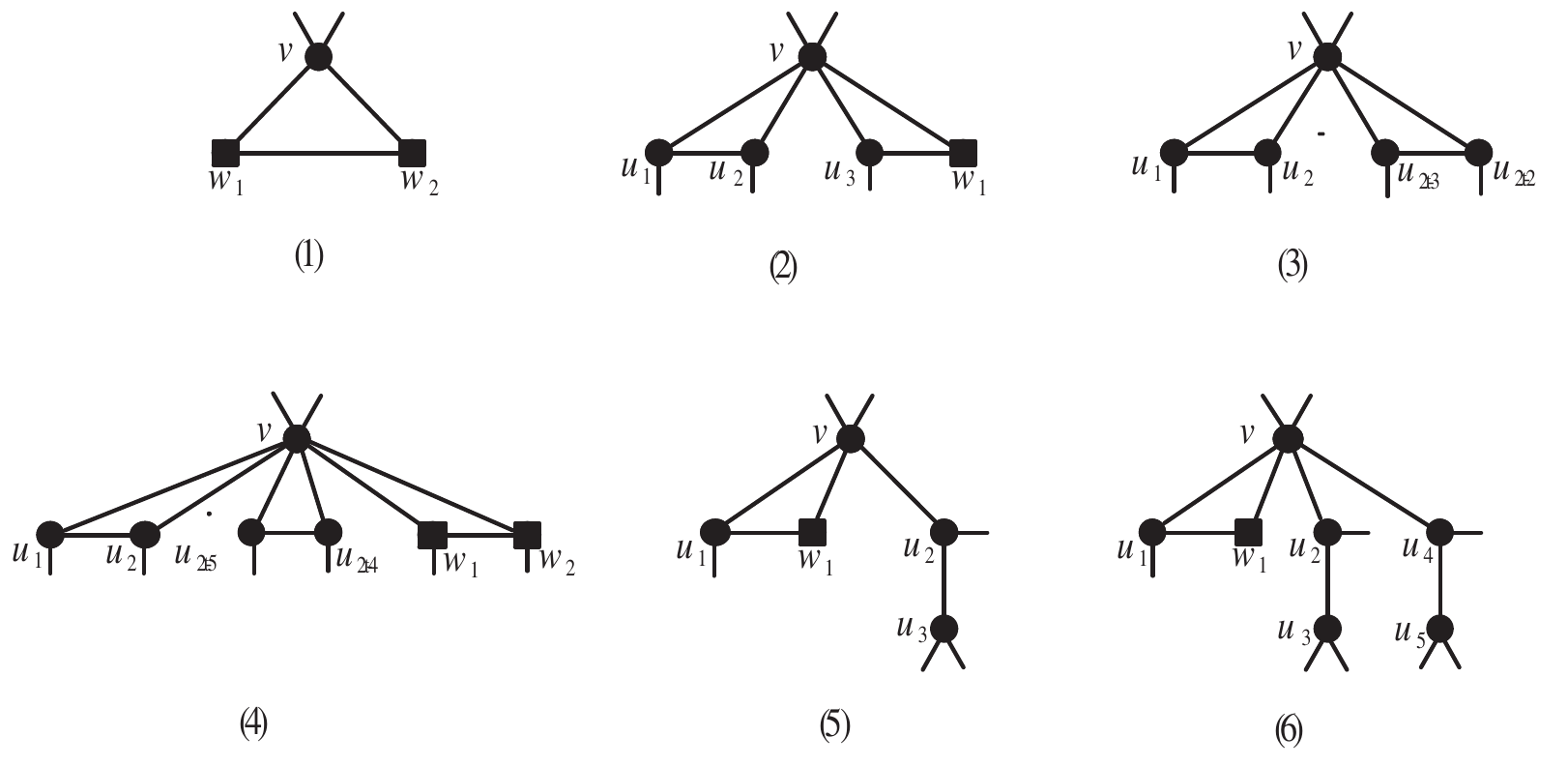}\\
{Figure 1: Situation in Lemma \ref{4redu60}}
\end{center}
\end{figure}
\begin{proof}
Let $H$ be the subgraph of $G$ induced by $\{v, u_i, v_i, w_i\}$ for all $i$ and all neighbors of bad vertices lying on any the worst or worse $3$-faces. Note that $|V(H)|\leq 2d(w)-1+6\leq29$. Consider any vertex $z\in V(H)$ and a $(\deg_{H}+\delta_{G,4})\downarrow z$-assignment $L$ for $H$. Since each block $B$ in $H$ is a $3$-face or an edge and then we can easily observe that $B$ satisfies (\textbf{FIX}) under $L$, by Lemma \ref{fix}, we know that $H$ satisfies (\textbf{FIX}).

Let $S$ be a $(P_3+P_4)$-independent subset in $G$ with $|S|\leq2$ and $L'$ a $(\deg_H+\delta-1_S)$-assignment for $H$. In $(1), (2), (3), (4)$, since every two vertices is connected by a path of length $2$ or $3$ in $H$, (\textbf{FORB}) is implied by (\textbf{FIX}). In $(5)$, if $S=\{v, u_2\}$ (or $S=\{u_2, u_3\}$), then we can first $L^{'}$-color $v, u_2, u_3$ (or $u_3, u_2, v$) and then greedily $L^{'}$-color the remaining vertices. If $S=\{s_1, u_3\}$, where $s_1$ lies on the worst or worse $3$-face $w_1s_1z$ incident with $w_1$, then we can first $L^{'}$-color $u_3, s_1, z, v, u_2$ in order and the greedily $L^{'}$-color the remaining vertices. In $(6)$, the result also holds by the same argument.

Hence, both (\textbf{FIX}) and (\textbf{FORB}) hold, and thus $H$ is $(P_3+P_4, 4)$-reducible.
\end{proof}

\begin{lemma}\label{redu2}
Let $G$ be a $\{C_4, C_5\}$-free planar graph, there are three $3$-faces $uvw, v_1v_2v_3, \\w_1w_2w_3$ that are incident with a $6$-face $f=vww_1w_3v_3v_1v$ such that $d(u)=3$, $d(v)=d(w)=d(v_1)=d(w_1)=4$, and $d(v_i)\leq12,d(w_i)\leq12$ for all $i\in\{2,3\}$. If each $x\in \{v_2,v_3,w_2,w_3\}$ satisfies that either $3\leq d(x)\leq4$ or $x$ is dangerous, then $G$ contains a $(P_3+P_4, 4)$-reducible induced subgraph with at most $49$ vertices.
\end{lemma}
\begin{figure}[H]
\begin{center}
\includegraphics[scale=0.8]{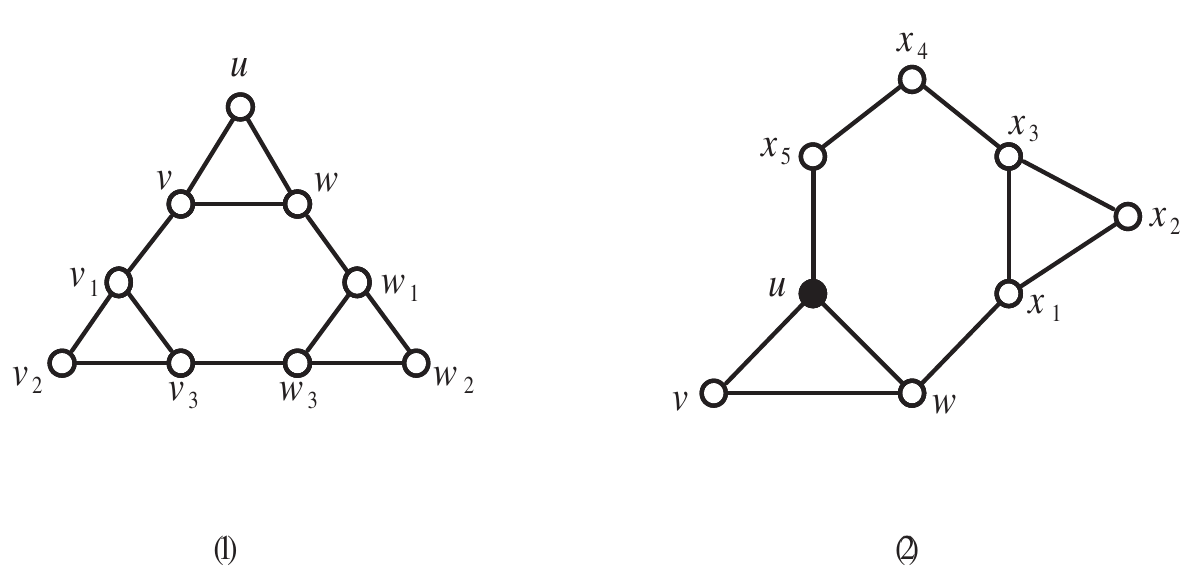}\\
{Figure 2: Situation in Lemma \ref{redu2} and Lemma \ref{redu3}}
\end{center}
\label{sub3}
\end{figure}
\begin{proof}
Let $H$ be the subgraph of $G$ induced by $A=\{u,v,w,v_1,v_2,v_3,w_1,w_2,w_3\}$ and all possible 3-neighbors inside a worst triangle incident with each $x\in \{v_2,v_3,w_2,w_3\}$, which is denoted by $B$ (possibly empty). Clearly, $V(H)=A\cup B$ and $|V(H)|\leq49$. Now we are ready to prove that $H$ is $(P_3+P_4, 4)$-reducible. It suffices to verify that $H$ satisfies (\textbf{FIX}) and (\textbf{FORB}). Firstly, consider an arbitrary $(\deg_{H}+\delta_{G,4})$-assignment $L$ for $H$. Then $|L(x)|>d_H(x)$ for each $x\in B\cup\{u\}$. In particular, $|L(x)|\geq d_{H'}(x)$ for each $x\in A$ and $|L(u)|>d_{H'}(u)$, where we denote $H'=H[A]$. It follows that (\textbf{FIX}) is immediately implied by Lemma \ref{fix}.

It remains to verify that for all $\{P_3+P_4\}$-independent set $S$, $H$ is $L'$-colorable for any $(\deg_H+\delta-1_S)$-assignment $L'$. By the assumption, we can easily observe that either $|S\cap A|\leq1$ or $S\in\{\{v,v_1\},\{w,w_1\},\{v_3,w_3\}\}$. The case when $S$ is a singleton is easily implied by (\textbf{FIX}). If $|S|=2$ and $S$ contains at least one $3$-vertex inside $B$, then we can first $L'$-color $H'$ by (\textbf{FIX}) and then greedily $L'$-color all vertices in $B$. Hence it remains to consider the case when $S\in\{\{v,v_1\},\{w,w_1\},\{v_3,w_3\}\}$. In all cases, we firstly $L'$-color $H[S]$ and let $\{a,b\}$ be the resulting colors on $S$. Let $H^*=H'-S$ and define a list assignment $L^*$ such that $L^*(z)=L'(z)\setminus\{a,b\}$ for each $z\in V(H')\cap N_H(S)$, while $L^*(z)=L'(z)$ for each $z\in V(H')\setminus N_H(S)$. It is easy to see that $|L^*(z)|\geq \deg_{H^*}(z)$ for all $z\in V(H^*)$. Since $|L^*(u)|>\deg_{H^*}(u)|$, by Lemma \ref{degree}, we can always $L^*$-color all vertices in $A-S$ and then extend the coloring to $B$ greedily. Hence $H$ satisfies (\textbf{FORB}).\medskip

This completes the proof of Lemma \ref{redu2}.
\end{proof}

\begin{lemma}\label{redu3}
Let $G$ be a $\{C_4, C_5\}$-free planar graph, if there are two $3$-faces $uvw, x_1x_2x_3$ that are adjacent to a $6$-face $uwx_1x_3x_4x_5u$ such that $d(u)=d(x_5)=3$ and $d(v)=d(w)=d(x_1)=4$. If each $x\in \{x_2,x_3,x_4\}$ satisfies that either $3\leq d(x)\leq4$ or $x$ is dangerous, then $G$ contains a $(P_3+P_4, 4)$-reducible induced subgraph with at most $37$ vertices.
\end{lemma}
\begin{proof}
Let $H$ be the subgraph of $G$ induced by $A=\{u,v,w,x_1,x_2,x_3,x_4,x_5\}$ and all possible 3-neighbors inside a worst triangle incident with each $x_i$ for all $i\in \{2,3,4\}$, which is denoted by $B$. Clearly, $V(H)=A\cup B$ and $|V(H)|\leq37$. Now we are ready to prove that $H$ is $(P_3+P_4, 4)$-reducible. It suffices to verify that $H$ satisfies (\textbf{FIX}) and (\textbf{FORB}). Consider an arbitrary $(\deg_{H}+\delta_{G,4})$-assignment $L$ of $H$, Then $|L(x)|>d_H(x)$ for each $x\in B\cup\{u,x_5\}$. In particular, $|L(x)|\geq d_{H'}(x)$ for each $x\in A$, where we denote $H'=H[A]$. It follows that (\textbf{FIX}) is immediately implied by Lemma \ref{fix}. %let $c$ be a color in $L(z)$ and let $L'$ be the list assignment for $H-z$ obtained from $L$ by removing $c$ from the lists of neighbors of $z$. We should prove that $H-z$ is $L'$-colorable. For any $z\in V(H)$, we can get that $H-z$ is connected and contains a vertex $z^{*}$ such that $|L(z^{*})|> \deg_{H-z}(z^{*})$, it follows that $H-z$ is $L'$-colorable, which implies (\textbf{FIX}) holds.

It remains to verify that for all $\{P_3+P_4\}$-independent set $S$, $H$ is $L'$-colorable for any $(\deg_H+\delta-1_S)$-assignment $L'$. By the assumption, we can easily observe that either $|S\cap A|\leq1$ or $S\in\{\{w,x_1\}, \{x_3,x_4\}, \{x_4,x_5\}, \{u,x_5\}\}$. The case when $S$ is a singleton is easily implied by (\textbf{FIX}). If $|S|=2$ and $S$ contains at least one $3$-vertex inside $B$, then we can first $L'$-color $H'$ by (\textbf{FIX}) and then greedily $L'$-color all vertices in $B$. Hence it only remains to consider the case when $S\in\{\{w,x_1\}, \{x_3,x_4\}, \{x_4,x_5\}, \{u,x_5\}\}$.

If $S=\{u,x_5\}$, then consider the subgraph $H'$ together with the list assignment $L'$, by Lemma \ref{degree}, we can always $L'$-color all vertices in $A$, and then greedily $L'$-color all vertices in $B$.
In the remaining cases, we firstly $L'$-color $H[S]$ and let $\{a,b\}$ be the resulting colors on $S$. Define a list assignment $L^*$ on $V(H)-S$ such that $L^*(z)=L'(z)\setminus\{a,b\}$ for each $z\in N_H(S)$, while $L^*(z)=L'(z)$ for each $z\in V(H)\setminus N_H(S)$. Since $|L^*(u)|=4, |L^*(x_5)|=3$, by Lemma \ref{degree}, we can always $L^*$-color all vertices in $A-S$ and then extend the coloring to $B$ greedily. Hence $H$ satisfies (\textbf{FORB}).
\end{proof}

%\begin{figure}[H]
%\begin{center}
%\includegraphics[scale=0.9]{combina.eps}\\
%{Figure 5: }
%\end{center}
%\label{sub5}
%\end{figure}

\medskip
\begin{figure}[H]
\begin{center}
\includegraphics[scale=0.7]{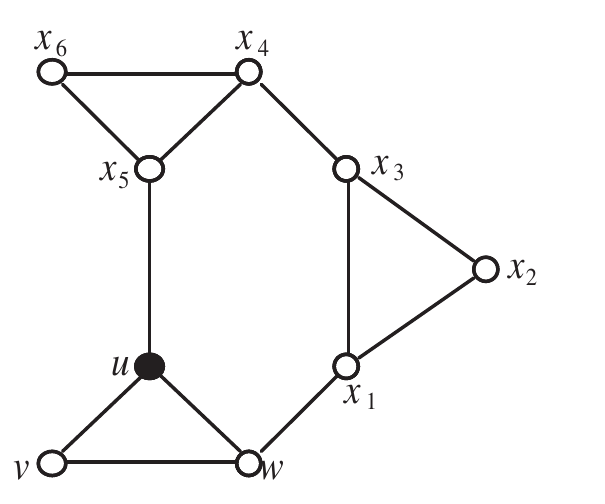}\\
{Figure 3: Situation in Lemma \ref{redu4}}
\end{center}
\end{figure}
\begin{lemma}\label{redu4}
Let $G$ be a $\{C_4, C_5\}$-free planar graph, if there are three $3$-faces $uvw$, $x_1x_2x_3$, $x_4x_5x_6$ that are adjacent to a $6$-face $uwx_1x_3x_4x_5u$ such that $d(u)=3$ and $d(v)=d(w)=d(x_1)=d(x_6)=4$ (see Figure $\mathrm{3}$). If $x_5$ is dangerous with $d(x_5)$ odd, while each $x\in \{x_2,x_3,x_4\}$ satisfies that either $3\leq d(x)\leq4$ or $x$ is dangerous, then $G$ contains a $(P_3+P_4, 4)$-reducible induced subgraph with at most $45$ vertices.
\end{lemma}

\begin{proof}
Let $H$ be the subgraph of $G$ induced by $A=\{u,v,w,x_1,x_2,x_3,x_4,x_5,x_6\}$ and all possible 3-neighbors inside a worst triangle incident with each $x_i$ for all $i\in \{2,3,4,5\}$, which is denoted by $B$. Clearly, $V(H)=A\cup B$ and $|V(H)|\leq45$. Now we are ready to prove that $H$ is $(P_3+P_4, 4)$-reducible. First, consider a $(\deg_{H}+\delta_{G,4})$-assignment $L$ for $H$, it is easy to observe that (\textbf{FIX}) is implied by Lemma \ref{fix}. %let $c$ be a color in $L(z)$ and let $L'$ be the list assignment for $H-z$ obtained from $L$ by removing $c$ from the lists of neighbors of $z$. We should prove that $H-z$ is $L'$-colorable. For any $z\in V(H)$, we can get that $H-z$ is connected and contains a vertex $z^{*}$ such that $|L(z^{*})|> \deg_{H-z}(z^{*})$, it follows that $H-z$ is $L'$-colorable, which implies (\textbf{FIX}) holds.

It remains to verify that for all $\{P_3+P_4\}$-independent set $S$, $H$ is $L'$-colorable for any $(\deg_H+\delta-1_S)$-assignment $L'$. The case when $S$ is a singleton is easily implied by (\textbf{FIX}), while the case when $S$ contains a 3-vertex inside $B$ is implied by (\textbf{FIX}) and Lemma \ref{degree}. Hence it suffices to consider $S\in\{\{w,x_1\}, \{x_3,x_4\}, \{u,x_5\}\}$.

If $S=\{u,x_5\}$, then consider the subgraph $H[A]$ together with the list assignment $L'$, by Lemma \ref{degree}, we can always $L'$-color all vertices in $A$, and then greedily $L'$-color all vertices in $B$.
In the remaining cases, we firstly $L'$-color $H[S]$ and let $\{a,b\}$ be the resulting colors on $S$. Let $H^*=H-S$ and define a list assignment $L^*$ such that $L^*(z)=L'(z)\setminus\{a,b\}$ for each $z\in V(H^*)\cap N_H(S)$, while $L^*(z)=L'(z)$ for each $z\in V(H^*)\setminus N_H(S)$. Since $|L^*(u)|=|L^*(x_5)|=4$, by Lemma \ref{degree}, we can always $L^*$-color all vertices in $A-S$ and then extend the coloring to $B$ greedily. Hence $H$ satisfies (\textbf{FORB}).

\end{proof}

%\begin{figure}[H]
%\begin{center}
%\includegraphics[scale=0.8]{local0.eps}\\
%{Figure 4: Situation in Lemma \ref{local}}
%\end{center}
%\end{figure}
\begin{figure}[H]
\begin{center}
\includegraphics[scale=0.8]{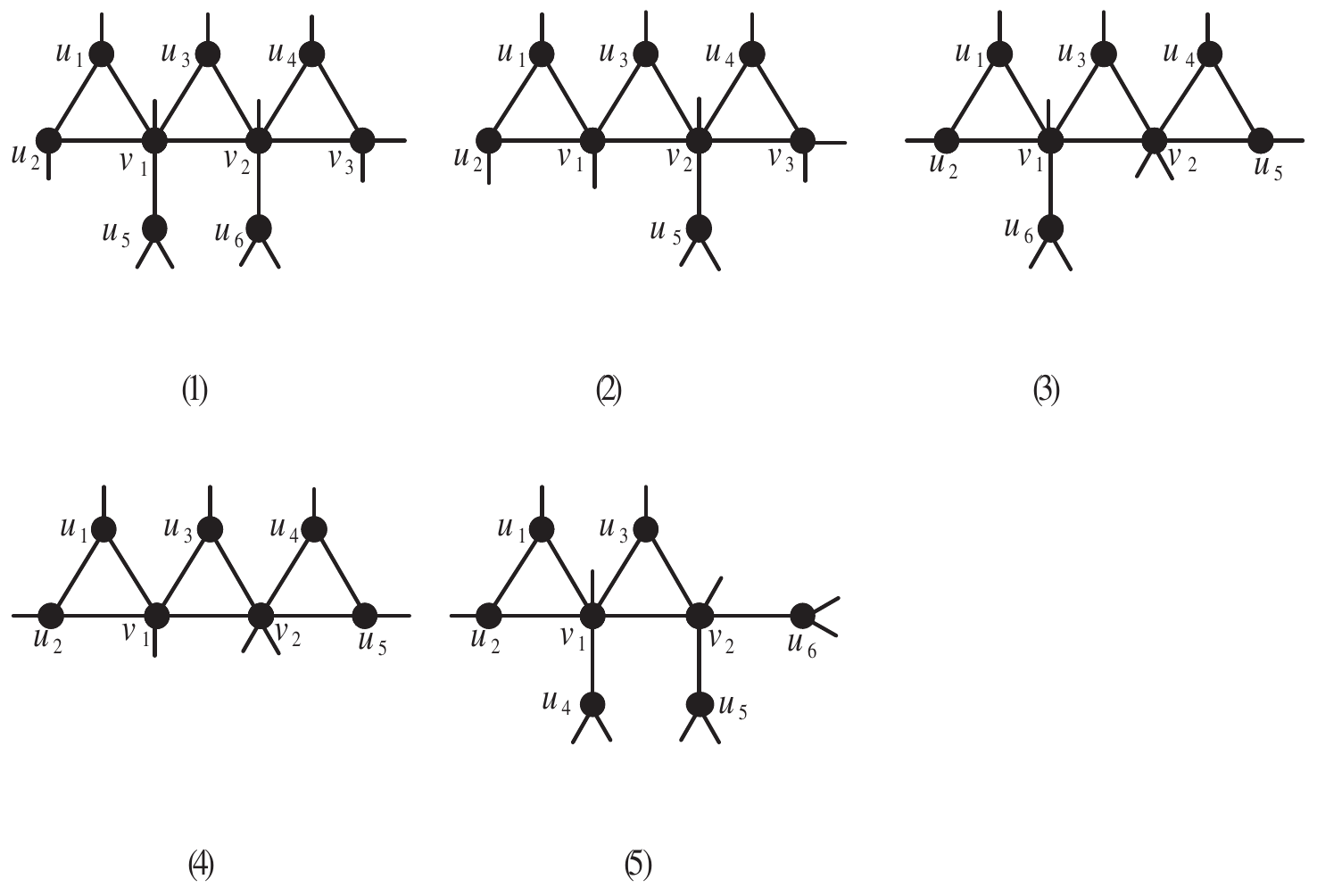}\\
{Figure 4: Situation in Lemma \ref{local}}
\end{center}
\end{figure}
\begin{lemma}\label{local}
If $G$ has a subgraph isomorphic to one of the configurations in Figure $4$, then $G$ contains a $(P_3+P_4, 4)$-reducible induced subgraph with at most $9$ vertices.
\end{lemma}

\begin{proof}
Let $H$ be the graph isomorphic one of the configurations in Figure 4. Note that (\textbf{FIX}) is easily obtained by Lemma \ref{fix}.
It remains to verify that for all $\{P_3+P_4\}$-independent set $S$, $H$ is $L'$-colorable for any $(\deg_H+\delta-1_S)$-assignment $L'$. The case when $S$ is a singleton is easily implied by (\textbf{FIX}). Now we only consider $|S|=2$. In all cases, $S$ induces a pendant edge in $H$ and it follows that we can first $L'$-color $H[S]$, and then greedily $L'$-color the remaining vertices in $H$.
Hence, (\textbf{FORB}) holds.
\end{proof}

\begin{figure}[H]
\begin{center}
\includegraphics[scale=0.8]{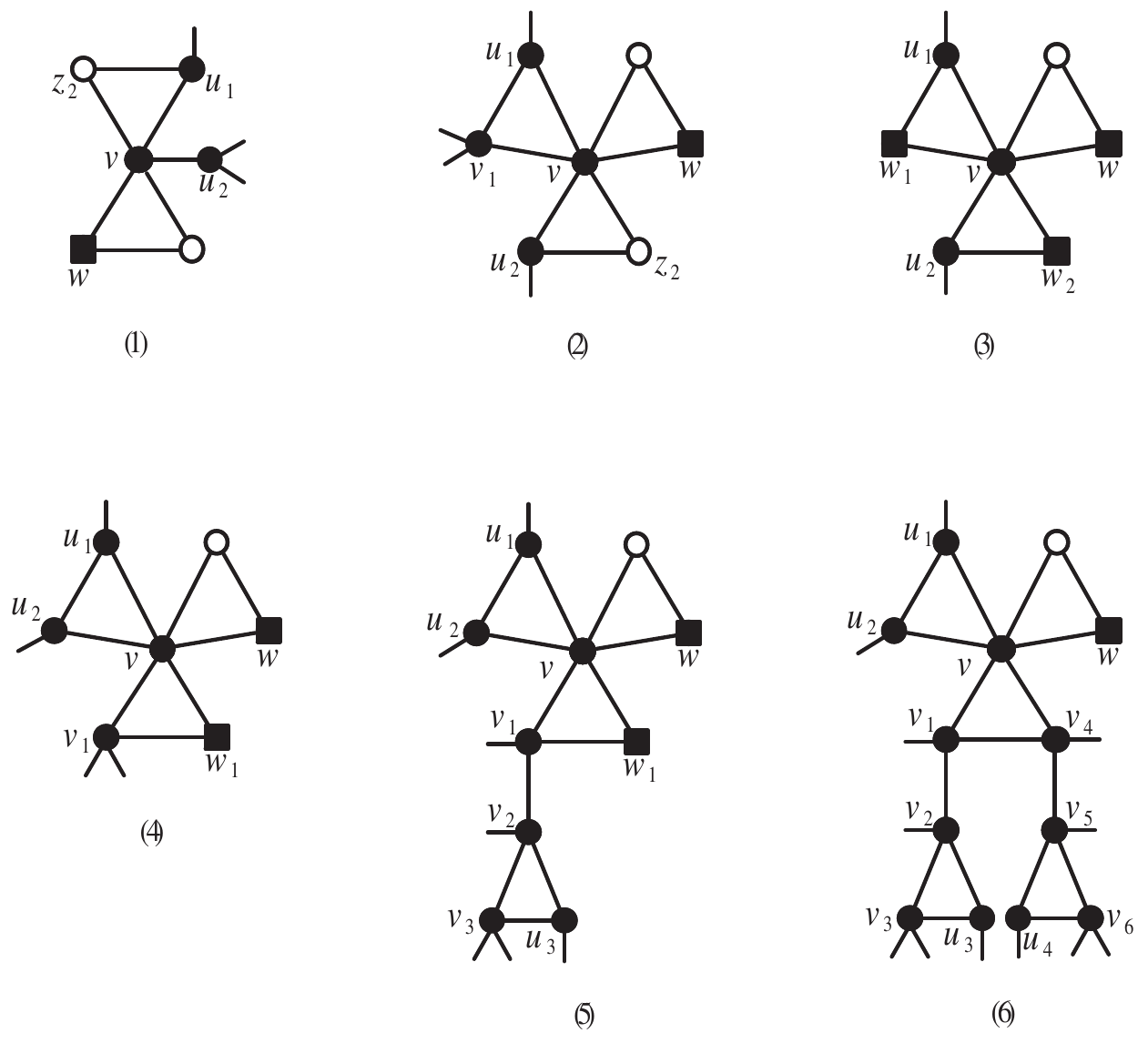}\\
{Figure 5: Situation in Lemma \ref{bad}}
\end{center}
\end{figure}
\begin{lemma}\label{bad}
Let $G$ be a planar graph satisfying any one of the following conditions (see Figure 5). Here $d(u_i)=3$, $d(v_i)=4$, $w_i$ is a bad $5^+$-vertex for all $i$, while $w$ is both bad and dangerous.
\begin{itemize}
\item[$(1)$] A $5$-vertex $v$ with $f_3(v)=2$, and $vu_2$ is a pendant edge, $vu_1x$ is a $3$-face, where $z_1$ is either a $4$-vertex or bad;
\item[$(2)$] A $6$-vertex $v$ with $f_3(v)=3$, and $vu_1v_1$, $vu_2y$ are $3$-faces, where $z_2$ is either a $4$-vertex or bad;
\item[$(3)$] A $6$-vertex $v$ with $f_3(v)=3$, and $f_{3b}(v)=2$;
\item[$(4)$] A $6$-vertex $v$ with $f_3(v)=3$, and $f_{3,3}(v)=1$, $f_{4b}(v)=1$;
\item[$(5)$] A $6$-vertex $v$ with $f_3(v)=3$, and $f_{4b}(v)=1$, two paths $v_1v_2v_3$, $v_2u_3v_3$ ($v_1$, $v_2$ may be the same vertex);
\item[$(6)$] A $6$-vertex $v$ with $f_3(v)=3$, and $f_{3,3}(v)=1$, four paths $vv_1v_2v_3$, $v_2u_3v_3$, $vv_4v_5v_6$, $v_5u_4v_6$ ($v_1$, $v_2$ may be the same vertex, $v_4$, $v_5$ may be the same vertex).
\end{itemize}
Then $G$ contains a $(P_3+P_4, 4)$-reducible induced subgraph with at most $37$ vertices.
\end{lemma}

\begin{proof}
We denote by $N^{*}(x)$ the set of neighbors of each bad vertex $x$ lying on the worse or worst $3$-faces. In all cases, if $d(w)$ is even, then by the assumption that $f_{3,3}(w)=\frac{d(w)-2}{2}$, Lemma \ref{4redu60}(3) directly implies a $(P_3+P_4, 4)$-reducible induced subgraph on at most $37$ vertices. Hence we may assume that $d(w)$ is odd.

Now consider the cases (1)-(4), let $H$ be the subgraph induced by $v, w$, all $u_i$, $v_i$, $w_i$, $z_i$ and all vertices in $N^{*}(x)$ for each bad vertex $x\in \{w,w_1,w_2, (z_i)\}$. Then we claim that $H$ is a $(P_3+P_4, 4)$-reducible induced subgraph. In fact, for any $(\deg_{H}+\delta_{G,4})$-assignment $L$ of $H$, we have $|L(w)|=2,|L(v)|\geq3$ and $|L(w_i)|\geq3, |L(v_i)|=2$ for each $i\in [2]$. In addition, $|L(y)|=3$ for all $y\in N^{*}(w)$, $|L(u_2)|=2$ in configuration (1), while $|L(u_i)|=3$ for each $i\in [2]$ in configurations (2)-(4). It follows from Lemma \ref{fix} that (\textbf{FIX}) holds. It remains to verify that for all $\{P_3+P_4\}$-independent set $S$, $H$ is $L'$-colorable for any $(\deg_H+\delta-1_S)$-assignment $L'$. The case when $S$ is a singleton is easily implied by (\textbf{FIX}). Now we only consider $|S|=2$. In all cases, we can first $L'$-color $H[S]$, and then greedily $L'$-color the remaining vertices in $H$. Hence, (\textbf{FORB}) holds.

Consider the configuration (5)-(6). Here, we just consider (6) since (5) can be solved by the same argument. Let $H$ be the subgraph induced by $v, w$, all $u_i,v_i$ and all vertices in $N^{*}(w)$. Now we claim that $H$ is a $(P_3+P_4, 4)$-reducible induced subgraph. In fact, for any $(\deg_{H}+\delta_{G,4})$-assignment $L$ of $H$, we have $|L(x)|=2$ for each $x\in \{w,v_3,v_6\}$, $|L(y)|=3$ for each $y\in \{v,u_1,u_2,u_3,u_4,v_1,v_2,v_4,v_5\}\cup N^{*}(w)$. It follows from Lemma \ref{fix} that (\textbf{FIX}) holds. It remains to verify that for all $\{P_3+P_4\}$-independent set $S$, $H$ is $L'$-colorable for any $(\deg_H+\delta-1_S)$-assignment $L'$. By similar arguments, in all cases, we can first $L'$-color $H[S]$, and then greedily $L'$-color the remaining vertices in $H$. Hence, (\textbf{FORB}) holds.
\end{proof}

%\begin{lemma}\label{34edge}
%Let $uv$ be an edge of $G$ such that $d(u)=3$ and $d(v)=4$. If there is a vertex $u_1$ and a $3$-face $v_1u_2v_2$ such that $d(u_1)=d(u_2)=3$, %$d(v_1)=d(v_2)=4$ and $uu_1$, $vv_1\in E(G)$. Then $G$ contains a $(P_3+P_4, 4)$-reducible induced subgraph with at most $6$ vertices.
%\end{lemma}
%\begin{proof}
%Let $H$ be the subgraph induced by $\{u, v, u_1, u_2, v_1, v_2\}$, consider any vertex $z\in V(H)$ and a $(\deg_{H}+\delta_{G,4})\downarrow z$-assignment $L$ for $H$, $H-z$ is $L$-colorable by Lemma \ref{degree}, implying (\textbf{FIX}). Let $L'$ be a $(\deg_H+\delta-1_S)$-assignment for $H$, If $S=\{u, u_1\}$ or $S=\{u, v\}$ or $S=\{v, v_1\}$, it is easy to greedily $L'$-color $H$. If $S=\{u_1, z\}$ where $z\in\{u_2, v_2\}$, we first consider the case $z=u_2$, then we can greedily $L'$-color $u_1, u_2, v_2, v_1, v, u$ in order. On the other hand, $z=v_2$, then we give a color to $v_2$ from $L^{'}(v_2)$ arbitrarily, and greedily $L'$-color the remaining vertices. Consequently, (\textbf{FORB}) admits. Hence, $H$ is $(P_3+P_4, 4)$-reducible.

%\end{proof}

\begin{lemma}\label{34edge}
Let $uv$ be an edge of $G$ such that $d(u)=3$ and $d(v)=4$. If there exists a vice vertex $v_2$ such that both $vv_1v_2$ and $v_2v_3u_1$ are $3$-faces incident with $v_2$ with $d(u_1)=3$ and $d(v_i)=4$ for each $i\in \{1, 2, 3\}$. Then $G$ contains a $(P_3+P_4, 4)$-reducible induced subgraph with at most $6$ vertices.
\end{lemma}

\begin{proof}
Let $H$ be the subgraph induced by $\{u, v, u_1, v_1, v_2, v_3\}$, consider any vertex $z\in V(H)$ and a $(\deg_{H}+\delta_{G,4})\downarrow z$-assignment $L$ for $H$, $H-z$ is $L$-colorable by Lemma \ref{degree}, implying (\textbf{FIX}). Now we only consider $|S|=2$. Since $S$ induces a pendant edge in $H$ and it follows that we can first $L'$-color $H[S]$, and then greedily $L'$-color the remaining vertices in $H$. Consequently, (\textbf{FORB}) admits. Hence, $H$ is $(P_3+P_4, 4)$-reducible.
\end{proof}

\section{Discharging Process}
\medskip
To prove Theorem \ref{thm3}, the main idea is to apply Lemma \ref{C4P4}. Let $G$ be a $\{C_4,C_5\}$-free planar graph. If $G$ satisfies the assumption in Lemma \ref{C4P4}, then we are done. Now we may assume that there exists $Z_0\subseteq V(G)$ such that $G[Z_0]$ does not contain any $(P_3+P_4, 4)$-reducible subgraph on at most $100$ vertices. Let $G'=G[Z_0]$. Since $G'$ is also a plane graph, by Euler's Formula, we obtain
\begin{equation*}
\sum_{v\in Z_0}(d(v)-2)+\sum_{f\in F(G')}(-2)=-4.
\end{equation*}

Now we redistribute the charges of all vertices and faces as follows, where we use $c(x\rightarrow y)$ to denote the charge sent from an element $x$ to another element $y$. Let $f$ be a face of $G'$ and $v_1,v_2,v_3$ be three vertices on $f$.
\begin{description}
\item[$\bf R0.$] If $f$ is a $3$-face and $d(v_1)\geq13$, then $c(v_1\rightarrow f)=\frac{4}{3}, c(v_2\rightarrow f)=c(v_3\rightarrow f)=\frac{1}{3}$;
\item[$\bf R1.$] For each $v\in V(f)$, if either $d(v)=3$, or $d(v)\geq4$ and $d(f)\geq6$, then $c(v\rightarrow f)=\frac{1}{3}$;
\item[$\bf R2.$] Let $f=v_1v_2v_3$ with $d(v_1)=4$.
\begin{description}
\item[$\bf R2.1.$] If $d(v_2)=4$, $d(v_3)=3$ and $v_1$ is not vice, then $c(v_1\rightarrow f)=1$, $c(v_2\rightarrow f)=\frac{2}{3}$;
\item[$\bf R2.2.$] If $d(v_2)=4$, $5\leq d(v_3)\leq12$ and $v_i$ is bad, $v_{3-i}$ is not vice for some $i\in\{1,2\}$, then $c(v_i\rightarrow f)=\frac{1}{2}$, $c(v_{3-i}\rightarrow f)=\frac{5}{6}$;
\item[$\bf R2.3.$] Otherwise, let $c(v_1\rightarrow f)=\frac{2}{3}$.
\end{description}
\item[$\bf R3.$] If $f=v_1v_2v_3$ with $3\leq d(v_1)\leq d(v_2)\leq4<d(v_3)\leq12$, then $c(v_3\rightarrow f)=2-c(v_1\rightarrow f)-c(v_2\rightarrow f)$.
\item[$\bf R4.$] Let $f=v_1v_2v_3$ such that $d(v_1)=3$, $5\leq d(v_2)\leq12$ and $5\leq d(v_3)\leq12$.
\begin{description}
\item[$\bf R4.1.$] If there exists a bad vertex $v_i\in\{v_2, v_3\}$, then $c(v_i\rightarrow f)=\frac{2}{3}$, $c(v_{5-i}\rightarrow f)=1$;
\item[$\bf R4.2.$] If there is no bad vertices on $f$, then $c(v_2\rightarrow f)=c(v_3\rightarrow f)=\frac{5}{6}$.
\end{description}
\item[$\bf R5.$] Let $f=v_1v_2v_3$ such that $d(v_1)=4$, $5\leq d(v_2)\leq12$, and $5\leq d(v_3)\leq12$.
\begin{description}
\item[$\bf R5.1.$] If there exists a bad vertex $v_i\in\{v_2, v_3\}$, then $c(v_i\rightarrow f)=\frac{1}{2}$, $c(v_{5-i}\rightarrow f)=\frac{5}{6}$;
\item[$\bf R5.2.$] If neither $v_2$ nor $v_3$ is bad, then $c(v_2\rightarrow f)=c(v_3\rightarrow f)=\frac{2}{3}$.
\end{description}
\item[$\bf R6.$] If $f=v_1v_2v_3$ such that $5\leq d(v_1)\leq d(v_2)\leq d(v_3)\leq12$,
    \begin{description}
        \item[$\bf R6.1.$] If there exists exactly one vertex which is both bad and dangerous, say $v_1$, then $c(v_1\rightarrow f)=\frac{1}{2}$, $c(v_2\rightarrow f)=c(v_3\rightarrow f)=\frac{3}{4}$;
        \item[$\bf R6.2.$] If there are two vertices, say $v_1,v_2$, where each $v_i$ is bad and dangerous, then $c(v_1\rightarrow f)=c(v_2\rightarrow f)=\frac{1}{2}$, $c(v_3\rightarrow f)=1$.
        \item[$\bf R6.3.$] Otherwise, let $c(v_i\rightarrow f)=\frac{2}{3}$ for each $i\in\{1,2,3\}$.
    \end{description}

\end{description}

Let $ch_1(x)$ be the new charge for each $x\in V(G')\cup F(G')$ after applying rules R1-R6. It is easy to see that for each vertex $v\in V(G')$, $f_3(v)\leq \lfloor\frac{d(v)}{2}\rfloor$. Since $G'$ does not contain any $(P_3+P_4, 4)$-reducible subgraph, Lemma \ref{CP} and Lemma \ref{4redu60} immediately imply the following corollary.
\begin{corollary}\label{CPC}
Each $3^+$-vertex $v\in V(G')$ satisfies the following:
\begin{description}
  \item[(1)] $f_{3,3}(v)+f_{3,4}(v)+f_{3b}(v)+f_{bb}(v)\leq\lfloor\frac{d(v)}{2}\rfloor-1$;
  \item[(2)] If $v$ is dangerous, then $f_{3b}(v)=f_{bb}(v)=0$. Furthermore, if $d(v)$ is odd, then $f_{3,4}(v)=f_{4b}(v)=0$
\end{description}
\end{corollary}

Moreover, we consider all $5^+$-vertices that satisfy certain properties as follows.
\begin{claim}\label{c1}
For each vertex $v$ of $G'$. If $d(v)\geq5$ and $v$ is not dangerous, then $ch_1(v)\geq\frac{1}{6}$. %\textcolor[rgb]{0.00,0.00,1.00}{If $v$ is not bad and $d(v)\geq7$, then $ch_1(v)\geq\frac{1}{6}$.(may be applied in the final analysis)}
\end{claim}
\begin{proof}
By Corollary \ref{CPC}(1), $f_{3,3}(v)+f_{3,4}(v)+f_{3b}(v)+f_{bb}(v)\leq\lfloor\frac{d(v)}{2}\rfloor-1$. %Since $v$ is not bad, from rules R1-R6, we can easily observe that \textcolor[rgb]{1.00,0.00,0.00}{when $d(v)=5$, $ch_1(v)\geq5-2-1-\frac{5}{6}-\frac{3}{3}\geq\frac{1}{6}$. When $d(v)=6$, we observe that if $f_{3,3}(v)=1$, then $f_{3b}(v)=f_{3,4}(v)=f_{bb}(v)=0$. Thus $ch_1(v)\geq6-2-\max\{\frac{4}{3}+\frac{5}{6}+\frac{5}{6}+\frac{3}{3},1+1+\frac{5}{6}+\frac{3}{3}\}\geq0$}. When $d(v)=7$, $ch_1(v)\geq7-2-\frac{4}{3}-1-\frac{5}{6}-\frac{4}{3}=\frac{1}{2}$. In general, for all $8^+$-vertex $v$, we have
%\begin{align}
%  ch_1(v)&\geq d(v)-2-\frac{4}{3}(\lfloor\frac{d(v)}{2}\rfloor-2)-1-\frac{5}{6}-\frac{1}{3}\lceil\frac{d(v)}{2}\rceil\nonumber\\
%  &\geq d(v)-2-\frac{d(v)}{3}+\frac{8}{3}-1-\frac{d(v)}{2}-\frac{5}{6}\nonumber\\
%  &=\frac{d(v)-7}{6}.\nonumber
%\end{align}
If $d(v)\geq5$ and $v$ is not dangerous, then $f_{3,3}(v)\leq\max\{0,\frac{d(v)+1}{2}-3\}$. Hence, when $d(v)=5$, $ch_1(v)\geq
5-2-1-\frac{5}{6}-\frac{3}{3}\geq\frac{1}{6}$. When $d(v)=6$, $ch_1(v)\geq
6-2-1-1-\frac{5}{6}-\frac{3}{3}\geq\frac{1}{6}$. When $d(v)\geq7$,
\begin{align}
  ch_1(v)&\geq d(v)-2-\frac{4}{3}\times(\frac{d(v)+1}{2}-3)-2-\frac{1}{3}\times\frac{d(v)+1}{2}\nonumber\\
  &\geq d(v)-\frac{5d(v)+5}{6}\nonumber\\
  &=\frac{d(v)-5}{6}\nonumber\\
  &\geq\frac{1}{3}.\nonumber
\end{align}
\end{proof}

\begin{claim}\label{c2}
For each $5^+$-vertex $v$ of $G'$. If $v$ is incident with a face $f=(u,v,w)$ in which $u$ is a bad $4$-vertex and $d(w)\geq4$, then $ch_1(v)\geq\frac{1}{6}$.
\end{claim}
\begin{proof}
If $d(v)\geq13$, then $ch_1(v)\geq d(v)-2-\frac{4}{3}\times(\lfloor\frac{d(v)}{2}\rfloor-1)-\frac{5}{6}-\frac{1}{3}\times(d(v)-\lfloor\frac{d(v)}{2}\rfloor)=\frac{2}{3}d(v)-\lfloor\frac{d(v)}{2}\rfloor-\frac{3}{2}\geq\frac{d(v)-9}{6}\geq\frac{1}{3}$ by R1 and R6. Next it suffices to consider the case when $d(v)\leq12$. Moreover, by Claim \ref{c1}, it suffices to consider the case when $v$ is dangerous.

By Corollary \ref{CPC}(2), since $f_{3,3}(v)=\lfloor\frac{d(v)-3}{2}\rfloor$, so $f_{bb}(v)=0$ and $w$ is not bad. %when $d(v)$ is odd, by Lemma \ref{}, we have $f_{3,3}(v)+f_{3b}(v)+f_{3,4}(v)+f_{bb}(v)\leq\frac{d(v)+1}{2}-2$. By R2-R5,
%\begin{align}
%  ch_1(v)&\geq d(v)-2-\frac{4}{3}(\frac{d(v)+1}{2}-3)-2\times\frac{5}{6}-\frac{1}{3}\frac{d(v)+1}{2}\nonumber\\
%  &\geq d(v)-\frac{5d(v)+5}{6}+\frac{2}{6}\nonumber\\
%  &=\frac{d(v)-3}{6}\nonumber\\
%  &\geq\frac{1}{3}.\nonumber
%\end{align}
%When $d(v)$ is even, by Lemma \ref{}, we have $f_{3,3}(v)+f_{3,4}(v)+f_{3b}(v)+f_{bb}(v)\leq\frac{d(v)}{2}-1$. \textcolor[rgb]{1.00,0.00,0.00}{If $d(v)=6$, then by Lemma \ref{}, we have $f_{4b}(v)=1$. Thus $ch_1(v)\geq6-2-\frac{4}{3}-\frac{5}{6}-\frac{2}{3}-\frac{3}{3}\geq\frac{1}{6}$.} If $d(v)\geq8$, then
%\begin{align}
%ch_1(v)
%  &\geq d(v)-2-\frac{4}{3}(\frac{d(v)}{2}-2)-2\times\frac{5}{6}-\frac{1}{3}\frac{d(v)}{2}\nonumber\\
%  &\geq d(v)-\frac{5d(v)}{6}-1\nonumber\\
%  &=\frac{d(v)-6}{6}\nonumber\\
%  &\geq\frac{1}{3}.\nonumber
%\end{align}
Then $c(v\rightarrow f)\leq\frac{2}{3}$. If $d(v)=5$, then $v$ is bad, and by Lemma \ref{4redu60}(1), we know that $d(w)\geq5$. Hence, by R5.1, $ch_1(v)\geq
5-2-\frac{4}{3}-\frac{1}{2}-\frac{3}{3}\geq\frac{1}{6}$.
When $d(v)=6$ and $f_{3,3}(v)=1$, if $v$ is bad, then $d(w)\geq5$ and by R5.1, $c(v\rightarrow f)=\frac{1}{2}$ and it follows that $ch_1(v)\geq
6-2-\frac{4}{3}-1-\frac{1}{2}-\frac{3}{3}\geq\frac{1}{6}$. Otherwise, since $f_{3,4}(v)+f_{3b}(v)+f_{bb}(v)=0$, $ch_1(v)\geq
6-2-\frac{4}{3}-\frac{5}{6}-\frac{2}{3}-\frac{3}{3}\geq\frac{1}{6}$.
When $d(v)\geq7$ is odd, by Corollary \ref{CPC}(2), we have $f_{3b}(v)+f_{3,4}(v)+f_{bb}(v)=0$ and
\begin{align}
  ch_1(v)&\geq d(v)-2-\frac{4}{3}\times(\frac{d(v)+1}{2}-2)-\frac{2}{3}-\frac{1}{3}\times\frac{d(v)+1}{2}\nonumber\\
  &\geq d(v)-\frac{5d(v)+5}{6}\nonumber\\
  &=\frac{d(v)-5}{6}\nonumber\\
  &\geq\frac{1}{3}.\nonumber
\end{align}
When $d(v)\geq8$ is even, we have $f_{3,3}(v)=\frac{d(v)}{2}-2$, $f_{3,4}(v)\leq1$. Hence,
\begin{align}
ch_1(v)&\geq d(v)-2-\frac{4}{3}\times(\frac{d(v)}{2}-2)-1-\frac{2}{3}-\frac{1}{3}\times\frac{d(v)}{2}\nonumber\\
  &\geq d(v)-\frac{5d(v)}{6}-1\nonumber\\
  &=\frac{d(v)-6}{6}\nonumber\\
  &\geq\frac{1}{3}.\nonumber
\end{align}
\end{proof}

A vertex $v$ in $G'$ is called \emph{well} when $ch_1(v)\geq \frac{1}{12}$. Given a poor face $f=(3,4,4)$ and a well vertex $v$.
From now on, let $\varpi(v)$ be the number of nice paths starting at $v$. For each poor face $f$, we apply the following rules.
\begin{description}
\item[$\bf R7.$] If $f$ is poor and $g$ is a $7^{+}$-face sharing an edge with $f$, then $f$ receives $(\frac{d(g)}{3}-2)/\xi(g)$ from $g$.
\item[$\bf R8.$] If $f$ receives less than $\frac{1}{3}$ by \textbf{R7} and $v_1, \ldots, v_t$ ($1\leq t\leq2$) are the well vertices such that all nice paths connecting each $v_i$ with $f$ has the same internal vertices, then $f$ receives $\frac{1}{6t}$ from each $v_i$ ($i\in \{1, \ldots, t\}$).
%Each well vertex sends $\frac{1}{6}$ to each poor face $f$ if $f$ receives less than $\frac{1}{3}$ by \textbf{R7} and there is a nice path connecting $v$ and $f$. In particular, if there are $t$ well vertices $\{v_1, \ldots, v_t\}$ connecting with $f$ in the same nice path, then $f$ receives $\frac{1}{6t}$ from each $v_i$ ($i\in \{1, \ldots, t\}$).
\end{description}
Let $ch_2(x)$ be the final charge for each element $x\in V(G')\cup F(G')$ after applying R7 and R8.
\subsection{Each poor face $f$ satisfies $ch_2(f)\geq0$.}\label{3poor}
\medskip

Let $f=(u,v,w)$ be a poor face such that $d(u)=3,d(v)=4,d(w)=4$ and $f_1,f_2,f_3$ be three adjacent faces sharing edge $vw, wu, uv$ with $f$, respectively. In addition, let $x_5$ be the neighbor of $u$ outside $f$.
If $v$ is not vice, then by \textbf{R1-R2}, $ch_2(f)\geq0$. Hence by symmetry, we may assume that both $v$ and $w$ are vice. Let $w_1,x_1$ and $v_1,y_1$ be the other two neighbors of $w$ and $v$ outside $f$, respectively. Then they are all $4^+$-vertices.

If there exists either a $5^+$-vertex or a $4$-vertex that is not vice, among $\{v_1,y_1,x_1,w_1\}$, say $x_1$, then by Claim \ref{c2}, $ch_1(x_1)\geq\frac{1}{6}$. In particular, if there are at least two such vertices as $x_1$ among $\{v_1,y_1,x_1,w_1\}$, then by R7-R8, we have $ch_2(f)\geq0$. Hence, without loss of generality, we further assume that at least three among them, say $x_1,w_1,v_1$, are vice $4$-vertices. Note that the vertex sets $\{v_2,v_3,w_2,w_3\}$ and $\{x_1,x_2,x_3,x_4\}$ do not overlap.

\medskip
\begin{claim}\label{c3}
If $d(f_1)=6$ and $x_1,y_1,w_1,v_1$ are vice $4$-vertices, then $f$ receives at least $\frac{1}{6}$ from a vertex in $\{v_2,v_3,w_2,w_3\}$.
\end{claim}
\begin{figure}[H]
\begin{center}
\includegraphics[scale=0.8]{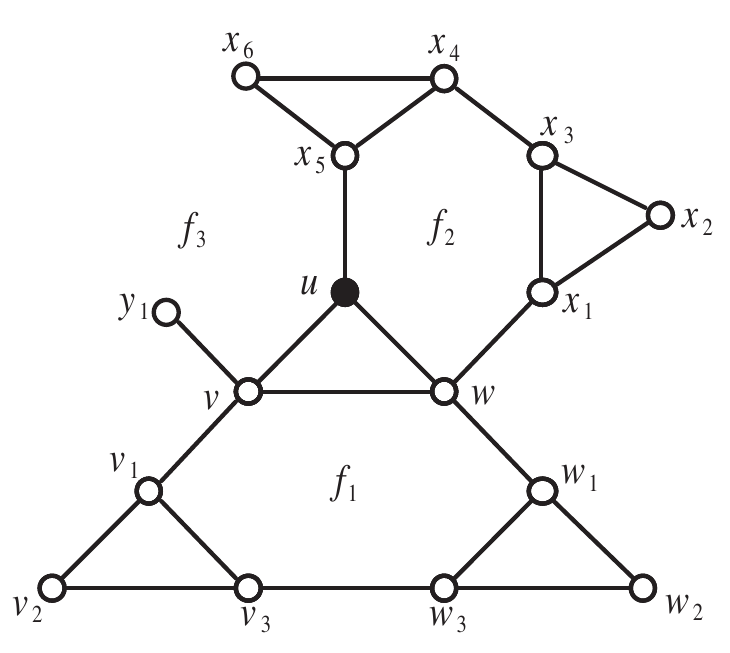}\\
{Figure 6}
\end{center}
\end{figure}
\begin{proof}
  We may assume that each vertex in $\{v_2,v_3,w_2,w_3\}$ has degree at most $12$.
Otherwise, by R8, $f$ receives at least $\frac{1}{6}$ in all from all $13^+$-vertices in $\{v_2,v_3,w_2,w_3\}$, and we are done. It is easily derived from Lemma \ref{redu2} that there exists a vertex in $\{v_2,v_3,w_2,w_3\}$, say $v_2$, such that $d(v_2)\geq5$ and $v_2$ is not dangerous.
By Claim \ref{c1}, $ch_1(v_2)\geq\frac{1}{6}$. Since there is a nice path connecting $v_2$ with $v$, by R8, $f$ receives $\frac{1}{6}$ from $v_2,v_3$ altogether.
\end{proof}

\begin{claim}\label{c4}
If $d(f_2)=6$, then $f$ receives at least $\frac{1}{6}$ from $x_2,x_3,x_4,x_5$ altogether.
\end{claim}
\begin{proof}
Suppose it is not true.
We may similarly assume that each vertex in $\{x_2,x_3,x_5\}$ has degree at most $12$.

If $d(x_5)=3$ and $d(x_4)\geq13$, then by R8, $f$ receives $\frac{1}{6}$ from $x_4$, a contradiction. If $d(x_5)=3$ and $d(x_4)\leq12$, then by Lemma \ref{redu3}, there exists a $5^+$-vertex in $\{x_2,x_3,x_4\}$ that is not dangerous. It follows from Claim \ref{c1} and R8 that $f$ receives $\frac{1}{6}$ from $x_2,x_3,x_4$ altogether.

If $d(x_5)=4$, then $x_5$ is not vice. Moreover, by Corollary \ref{CPC}(1), $f_{3,4}(v)=0$, and it is easy to see that $x_5$ is well. So $f$ receives $\frac{1}{6}$ from $x_5$.

If $5\leq d(x_5)\leq12$, then by Claim \ref{c1}, it remains to consider the case when $x_5$ is dangerous. Now we have the following observations:
\begin{enumerate}
  \item[(1)] $f_{3b}(x_5)=f_{bb}(x_5)=0$;
  \item[(2)] $x_4x_5$ is contained in a $3$-face, say $x_4x_5x_6$;
  \item[(3)] $x_4,x_6$ are $4^+$-vertices;
  \item[(4)] $d(x_5)$ is odd.
\end{enumerate}
In fact, $(1)$ and $(3)$ are immediately derived from Corollary \ref{CPC}. For $(2)$, if $x_4x_5$ is not contained in a $3$-face, then $x_5$ is incident with at most $\frac{d(v)-2}{2}$ $3$-faces. Since $x_5$ is dangerous, it is easy to see that if $d(x_5)$ is odd, then by Corollary \ref{CPC}(2), $f_{3,4}(x_5)=0$ and $ch_1(x_5)\geq d(x_5)-2-\frac{4}{3}\times\frac{d(x_5)-3}{2}-\frac{1}{3}\times\lceil\frac{d(x_5)+1}{2}\rceil\geq\frac{1}{6}$, a contradiction. If $d(x_5)$ is even, then $ch_1(x_5)\geq d(x_5)-2-\frac{4}{3}\times\frac{d(x_5)-4}{2}-1-\frac{1}{3}\times\frac{d(x_5)+2}{2}\geq\frac{1}{3}$, a contradiction.
For $(4)$, If $d(x_5)$ is even, then since $ux_5$ is not contained in a $3$-face, we know $x_5$ is incident with at most $\frac{d(v)-2}{2}$ $3$-faces. Applying similar arguments as
above, we know that $ch_1(x_5)\geq d(x_5)-2-\frac{4}{3}\times\frac{d(x_5)-4}{2}-1-\frac{1}{3}\times\frac{d(x_5)+2}{2}\geq\frac{1}{3}$. In all those cases, $f$ receives $\frac{1}{6}$ from $x_5$, a contradiction.

By the above observations, we know $x_5$ is also a bad vertex. Consider the $3$-face $f'=x_4x_5x_6$, if $d(x_i)\geq5$ for all $i\in\{4,6\}$, then by R6, $c(x_5\rightarrow f')=\frac{1}{2}$ and it follows that $ch_1(x_5)\geq d(x_5)-2-\frac{4}{3}\times\frac{d(x_5)-3}{2}-\frac{1}{2}-\frac{1}{3}\times\frac{d(x_5)+1}{2}\geq\frac{1}{6}$. If $d(x_i)=4$ for some $i\in\{4,6\}$, then $x_{10-i}$ is not bad. If $d(x_{10-i})\geq13$, then by R0, $c(x_5\rightarrow f')=\frac{1}{3}$. We know that $ch_1(x_5)\geq d(x_5)-2-\frac{4}{3}\times\frac{d(x_5)-3}{2}-\frac{1}{3}-\frac{1}{3}\times\frac{d(x_5)+1}{2}\geq\frac{1}{3}$, a contradiction. If $5\leq d(x_{10-i})\leq12$, then by R5.1, $c(x_5\rightarrow f')=\frac{1}{2}$, and it follows that $ch_1(x_5)\geq d(x_5)-2-\frac{4}{3}\times\frac{d(x_5)-3}{2}-\frac{1}{2}-\frac{1}{3}\times\frac{d(x_5)+1}{2}\geq\frac{1}{6}$, a contradiction. Now it remains to consider the case when $d(x_4)=d(x_6)=4$. By Lemma \ref{redu4}, there is a $5^+$-vertex inside $\{x_2,x_3\}$ that is not dangerous, say $x_2$, and by Claim \ref{c1}, $f$ receives $\frac{1}{6}$ from $x_2,x_3$ altogether, a contradiction.
\end{proof}

Now we proceed to verify that $ch_2(f)\geq0$. If there are at least two $7^+$-faces among $\{f_1,f_2,f_3\}$, then $ch_2(f)\geq0$ by R7. Hence we assume at least two of them are $6$-face.
By the assumption that $x_1,w_1,v_1$ are vice $4$-vertices, If $y_1$ is not vice and $f_2$ is $7^+$-face, then we are done. Otherwise, if $y_1$ is vice, then by Claim \ref{c3}, $f$ receives at least $\frac{1}{6}$ from a vertex in $\{v_2,v_3,w_2,w_3\}$. Meanwhile, if $f_2$ is a $6$-face, then by Claim \ref{c4}, $f$ receives at least $\frac{1}{6}$ from a vertex in $\{x_2,x_3,x_4,x_5\}$. This completes the proof.\medskip

\subsection{Final analysis}
Now we shall verify that $ch_2(x)\geq0$ for all $x\in V(G')\cup F(G')$. It is easy to observe that for each $v\in V(G')$, if $f_3(v)<\lfloor\frac{d(v)}{2}\rfloor$, then $ch_2(v)\geq0$. From now on, it suffices to only care the case $f_3(v)=\lfloor\frac{d(v)}{2}\rfloor$. We assume that each bad vertex satisfies $d(v)\geq5$. For $v\in V(G')$, denote by $N_b(v)$ the set of bad neighbors of $v$ and let $|N_b(v)|=n_b(v)$.
\begin{claim}\label{np1}
Let $v$ be a vertex of $G'$ with $5\leq d(v)\leq 12$. If $v$ is not dangerous, then $ch_2(v)\geq0$.
\end{claim}

\begin{proof}
If $v$ is not dangerous, it follows that $f_{3,3}(v)\leq \lfloor\frac{d(v)-3}{2}\rfloor-1$. Accordingly,
\begin{align}
  ch_2(v)&\geq d(v)-2-\frac{4}{3}f_{3,3}(v)-f_{3,4}(v)-f_{3b}(v)-\frac{2}{3}f_{4,4}(v)-\frac{5}{6}f_{4b}(v)-f_{bb}(v)\nonumber\\
   &-\frac{1}{6}f_{3,4}(v)-\frac{1}{6}f_{3b}(v)-2\times \frac{1}{6}f_{4,4}(v)-\frac{1}{6}f_{4b}(v)-\frac{1}{3}(d(v)-f_3(v))\nonumber\\
  &\geq \frac{2}{3}d(v)-\frac{2}{3}f_3(v)-\frac{1}{3}f_{3,3}(v)-\frac{1}{6}(f_3(v)-f_{3,3}(v)-1)-2\nonumber\\
  &\geq \frac{d(v)-8}{6}.\nonumber  \tag{*}
\end{align}
Thus, $ch_2(v)\geq0$ when $d(v)\geq8$.

In particular, when $d(v)=7$, we get $ch_2(v)\geq \frac{2}{3}\times7-\frac{2}{3}\times3-\frac{1}{3}\times1-\frac{1}{6}\times1-2=\frac{1}{6}>0$.

\textbf{When $d(v)=6$.} Let $f_1=v_1v_2v$, $f_2=v_3v_4v$ and $f_3=v_5v_6v$ be three $3$-faces incident with $v$, respectively. We discuss the following two cases depending on whether $v$ is bad:\medskip

$\bullet$ $v$ is bad but not dangerous in $G^{'}$;%, then $ch_2(v)\geq0$.

\noindent
We may assume that $f_2$ and $f_3$ are worse, and it follows from $\{\mathrm{(a)}$, $\mathrm{(d)}$, $\mathrm{(d)}\}$ that $d(v_1)\geq4$, $d(v_2)\geq4$.
If $d(v_1)=d(v_2)=4$, then $\varpi(v)\leq2$ by $\{\mathrm{(e)}, \mathrm{(e)}, \mathrm{(k)}\}$ and Lemma \ref{34edge}, thus $ch_2(v)\geq6-2-3\times \frac{1}{3}-2\times1-\frac{2}{3}-2\times \frac{1}{6}=0$ by R3 and R8.
If $d(v_1)=4$, $d(v_2)\geq5$, then by $\{\mathrm{(d)}, \mathrm{(d)}, \mathrm{(i)}\}$, we get $n_b(v)=0$ and $\varpi(v)\leq3$ by Lemma \ref{34edge}. Thus $ch_2(v)\geq6-2-3\times\frac{1}{3}-2\times1-\frac{1}{2}-3\times\frac{1}{6}=0$ by R5.1.
If $d(v_1)\geq5$, $d(v_2)\geq5$, then $n_b(v)\leq1$ by $\{\mathrm{(d)}, \mathrm{(d)}, \mathrm{(h)}\}$. If there exists a bad vertex, then it can not be dangerous by Lemma \ref{bad}(2), thus $ch_2(v)\geq6-2-3\times \frac{1}{3}-2\times1-\frac{2}{3}-2\times\frac{1}{6}=0$ by R6.3.\medskip

$\bullet$ $v$ is neither bad nor dangerous in $G^{'}$;%, then $ch_2(v)\geq0$.

\textbf{Case 1.} $v$ is incident with a worse face;

\noindent
W.l.o.g, let $f_3$ be worse and $N_1(v)=N(v)\setminus \{v_5, v_6\}$. Suppose there are two $3$-vertices in $N_1(v)$, say $v_1$ and $v_3$, then $d(v_2)\geq5$, $d(v_4)\geq5$.
%By $\{\mathrm{(d)}, \mathrm{(f)}, \mathrm{(f)}\}$, $n_b(v)\leq1$. If $n_b(v)=1$, then $\varpi(v)\leq1$ by $\{\mathrm{(b)}, \mathrm{(d)}, \mathrm{(g)}\}$, $\{\mathrm{(b)}, \mathrm{(e)}, \mathrm{(f)}\}$, $\{\mathrm{(a)}, \mathrm{(e)}, \mathrm{(g)}\}$, thus $ch_2(v)\geq 6-2-3\times \frac{1}{3}-1-1-\frac{5}{6}-\frac{1}{6}=0$ by R4. Otherwise
By $\{\mathrm{(a)}, \mathrm{(d)}, \mathrm{(f)}\}$, we obtain $n_b(v)=0$, by Lemma \ref{local}, we get that both $v_2$ and $v_4$ are well vertices, thus $ch_2(v)\geq 6-2-3\times \frac{1}{3}-1-2\times \frac{5}{6}-\frac{1}{6}-2\times\frac{1}{12}=0$ by R8. Suppose there are only one $3$-vertex in $N_1(v)$, say $d(v_2)=3$, then $d(v_1)\geq5$. If $d(v_3)=d(v_4)=4$, we first consider $v_1$ is bad, then $\varpi(v)\leq2$ by $\{\mathrm{(e)}, \mathrm{(g)}, \mathrm{(k)}\}$, $\{\mathrm{(e)}, \mathrm{(f)}, \mathrm{(l)}\}$, $\{\mathrm{(d)}, \mathrm{(g)}, \mathrm{(l)}\}$, thus $ch_2(v)\geq 6-2-3\times \frac{1}{3}-1-1-\frac{2}{3}-2\times\frac{1}{6}=0$ by R3 and R4. Otherwise, $v_1$ is not bad and it follows from $\{\mathrm{(b)}, \mathrm{(d)}, \mathrm{(l)}\}$ that $\varpi(v)\leq3$, thus $ch_2(v)\geq 6-2-3\times \frac{1}{3}-1-\frac{5}{6}-\frac{2}{3}-3\times\frac{1}{6}=0$ by R4. If there exists a $5^{+}$-vertex in $\{v_3, v_4\}$, say $v_3$, it follows from $\{\mathrm{(d)}, \mathrm{(f)}, \mathrm{(i)}\}$ that $n_b(v)\leq1$. If $v_1$ is bad, then $\varpi(v)\leq2$ by $\{\mathrm{(e)}, \mathrm{(g)}, \mathrm{(m)}\}$. On the other hand, if $v_3$ is bad, we can also get $\varpi(v)\leq2$ by $\{\mathrm{(b)}, \mathrm{(d)}, \mathrm{(j)}\}$. If there is no bad vertex in $N_1(v)$, then $\varpi(v)\leq3$ by Lemma \ref{34edge}. Thus $ch_2(v)\geq 6-2-3\times \frac{1}{3}-1-\max\{1+\frac{2}{3}+2\times\frac{1}{6}, 2\times\frac{5}{6}+2\times\frac{1}{6}, \frac{5}{6}+\frac{2}{3}+3\times \frac{1}{6}\}=0$ by R4 and R5. If $d(v_3)\geq5$, $d(v_4)\geq5$, then it follows from $\{\mathrm{(a)}, \mathrm{(d)}, \mathrm{(h)}\}$ that $v_3$ and $v_4$ are not bad at the same time, thus $n_b(v)\leq2$. If $n_b(v)=2$, then by Lemma \ref{bad}(2),some $v_i$ for $i\in\{3,4\}$ is bad but not dangerous, and it follows that $ch_2(v)\geq 6-2-3\times \frac{1}{3}-1-1-\frac{2}{3}-2\times\frac{1}{6}=0$ by R4 and R6. If $n_b(v)\leq1$, then $ch_2(v)\geq 6-2-3\times \frac{1}{3}-1-\max\{1+\frac{2}{3}+2\times\frac{1}{6}, \frac{5}{6}+\frac{3}{4}+2\times\frac{1}{6}\}=0$ by R4, R6, and R8.
Next, we consider $d(z)\geq4$ for all $z\in N_1(v)$. If $d(z)=4$ for all $z\in N_1(v)$, then $\varpi(v)\leq4$ by $\{\mathrm{(d)}, \mathrm{(l)}, \mathrm{(l)}\}$, thus $ch_2(v)\geq6-2-3\times \frac{1}{3}-1-2\times\frac{2}{3}-4\times\frac{1}{6}=0$ by R3. Suppose there exists a $5^{+}$-vertex in $N_1(v)$, say $v_1$. If $v_1$ is bad, then $\varpi(v)\leq3$ by $\{\mathrm{(d)}, \mathrm{(j)}, \mathrm{(l)}\}$, thus $ch_2(v)\geq6-2-3\times \frac{1}{3}-1-\frac{2}{3}-\frac{5}{6}-3\times\frac{1}{6}=0$ by R5.1. If $v_1$ is not bad, then $ch_2(v)\geq6-2-3\times \frac{1}{3}-1-2\times\frac{2}{3}-4\times\frac{1}{6}=0$ by R5.2. Suppose there exists two $5^{+}$-vertices in $N_1(v)$, say $v_1$ and $v_2$ or $v_1$ and $v_3$. If $n_b(v)=2$, then $\varpi(v)\leq2$ by $\{\mathrm{(d)}, \mathrm{(j)}, \mathrm{(j)}\}$ or $\{\mathrm{(d)}, \mathrm{(h)}, \mathrm{(l)}\}$, thus $ch_2(v)\geq6-2-3\times \frac{1}{3}-1-\max\{2\times\frac{5}{6}, 1+\frac{2}{3}\}-2\times\frac{1}{6}=0$ by R3, R5 and R6. If $n_b(v)\leq1$, then $ch_2(v)\geq6-2-3\times \frac{1}{3}-1-\max\{\frac{3}{4}+\frac{2}{3}, \frac{2}{3}+\frac{5}{6}\}-3\times\frac{1}{6}=0$ by R5 and R6. If there exists three $5^{+}$-vertices in $N_1(v)$, then it follows from $\{\mathrm{(d)}, \mathrm{(h)}, \mathrm{(i)}\}$ that $n_b(v)\leq2$ and $ch_2(v)\geq6-2-3\times\frac{1}{3}-1-\max\{1+\frac{2}{3}, \frac{5}{6}+\frac{3}{4}\}-2\times\frac{1}{6}=0$ by R5 and R6. If $d(z)\geq5$ for all $z\in N_1(v)$, then it follows from $\{\mathrm{(d)}, \mathrm{(h)}, \mathrm{(h)}\}$ that $n_b(v)\leq3$ and $ch_2(v)\geq6-2-3\times\frac{1}{3}-1-\frac{3}{4}-1-\frac{1}{6}=\frac{1}{12}>0$ by R6.\medskip

\textbf{Case 2.} $v$ is not incident with a worse face;

\noindent
If $n_3(v)=3$, then $n_b(v)\leq1$ by $\{\mathrm{(a)}, \mathrm{(f)}, \mathrm{(f)}\}$. In particular, $\varpi(v)\leq2$ when $n_b(v)=1$ by Lemma \ref{4redu60}(6), thus $ch_2(v)\geq6-2-3\times\frac{1}{3}-\max\{1+2\times\frac{5}{6}+2\times\frac{1}{6}, 3\times \frac{5}{6}+3\times \frac{1}{6}\}=0$ by R4.
If $n_3(v)=2$, w.l.o.g, we say $d(v_1)=d(v_3)=3$, then $d(v_2)\geq5$, $d(v_4)\geq5$. Suppose $d(v_5)=d(v_6)=4$, if $n_b(v)=2$, then $\varpi(v)\leq2$ by $\{\mathrm{(g)}, \mathrm{(g)}, \mathrm{(k)}\}$, and it follows that $ch_2(v)\geq6-2-3\times\frac{1}{3}-2\times1-\frac{2}{3}-2\times\frac{1}{6}=0$ by R3 and R4. If $n_b(v)\leq1$, then $ch_2(v)\geq6-2-3\times\frac{1}{3}-\max\{1+\frac{5}{6}+\frac{2}{3}+3\times\frac{1}{6}, 2\times\frac{5}{6}+\frac{2}{3}+4\times\frac{1}{6}\}=0$. If there is a $5^{+}$-vertex in $\{v_5, v_6\}$, say $v_5$, then $n_b(v)\leq2$ by $\{\mathrm{(f)}, \mathrm{(f)}, \mathrm{(i)}\}$. If $n_b(v)=2$, then $\varpi(v)\leq2$ by $\{\mathrm{(g)}, \mathrm{(g)}, \mathrm{(m)}\}$ or $\{\mathrm{(b)}, \mathrm{(g)}, \mathrm{(j)}\}$, thus $ch_2(v)\geq6-2-3\times\frac{1}{3}-\max\{2\times1+\frac{2}{3}, 1+\frac{5}{6}+\frac{5}{6}\}-2\times\frac{1}{6}=0$ by R4 and R5. Otherwise $n_b(v)\leq1$, we get $ch_2(v)\geq6-2-3\times\frac{1}{3}-\max\{1+\frac{5}{6}+\frac{2}{3}, 3\times\frac{5}{6}\}-3\times\frac{1}{6}=0$ by R4 and R5. If both of $v_5$ and $v_6$ are $5^{+}$-vertex, then $n_b(v)\leq3$ by $\{\mathrm{(f)}, \mathrm{(f)}, \mathrm{(h)}\}$. If $n_3(b)=3$, then $v_i$ $(i\in\{5,6\})$ is bad but not dangerous by Lemma \ref{bad}(3), thus $ch_2(v)\geq6-2-3\times\frac{1}{3}-2\times1-\frac{2}{3}-2\times\frac{1}{6}=0$ by R6. Otherwise $n_3(b)\leq2$, we have $ch_2(v)\geq6-2-3\times\frac{1}{3}-\max\{1+\frac{3}{4}+\frac{5}{6}, 1+2\times\frac{5}{6}, 1\times2+\frac{2}{3}\}-2\times\frac{1}{6}=0$ by R4 and R6. If $n_3(v)=1$, w.l.o.g, say $d(v_1)=3$, then $d(v_2)\geq5$, we denote $N_2(v)=N(v)\backslash\{v_1, v_2\}$. If $d(z)=4$ for all $z\in N_2(v)$. Suppose $n_b(v)=1$, then $\varpi(v)\leq3$ by $\{\mathrm{(f)}, \mathrm{(k)}, \mathrm{(l)}\}$, thus $ch_2(v)\geq6-2-3\times\frac{1}{3}-1-2\times\frac{2}{3}-3\times\frac{1}{6}=\frac{1}{6}>0$ by R3 and R4. Otherwise $n_b(v)=0$, it follows that $ch_2(v)\geq6-2-3\times\frac{1}{3}-\frac{5}{6}-2\times\frac{2}{3}-5\times\frac{1}{6}$=0. If there exists an $5^{+}$-vertex in $N_2(v)$, say $v_3$. If $n_b(v)=2$, then $\varpi(v)\leq3$ by $\{\mathrm{(g)}, \mathrm{(j)}, \mathrm{(l)}\}$, thus $ch_2(v)\geq6-2-3\times\frac{1}{3}-1-\frac{5}{6}-\frac{2}{3}-3\times\frac{1}{6}=0$ by R4 and R5. Otherwise $n_b(v)\leq1$, we get $ch_2(v)\geq6-2-3\times\frac{1}{3}-\max\{1+2\times\frac{2}{3}, 2\times\frac{5}{6}+\frac{2}{3}\}-4\times\frac{1}{6}=0$ by R4 and R6. If there are two $5^{+}$-vertices in $N_2(v)$, say $v_3$ and $v_4$, or $v_3$ and $v_5$. In the former case, if $n_b(v)=3$, then $\varpi(v)\leq1$ by $\{\mathrm{(f)}, \mathrm{(h)}, \mathrm{(k)}\}$, thus $ch_2(v)\geq6-2-3\times\frac{1}{3}-2\times1-\frac{2}{3}-\frac{1}{6}=\frac{1}{6}>0$. If $n_b(v)\leq2$, then $ch_2(v)\geq6-2-3\times\frac{1}{3}-\max\{1+\frac{5}{6}+\frac{2}{3}, 1+\frac{3}{4}+\frac{2}{3}\}-3\times\frac{1}{6}=0$. In the latter case, if $n_b(v)=3$, then $\varpi(v)\leq2$ by $\{\mathrm{(g)}, \mathrm{(j)}, \mathrm{(j)}\}$, thus $ch_2(v)\geq6-2-3\times\frac{1}{3}-1-2\times\frac{5}{6}-2\times\frac{1}{6}=0$. If $n_b(v)\leq2$, then $ch_2(v)\geq6-2-3\times\frac{1}{3}-\max\{1+\frac{5}{6}+\frac{2}{3}, 3\times\frac{5}{6}\}-3\times\frac{1}{6}=0$. If there are three $5^{+}$-vertices in $N_2(v)$, then $n_b(v)\leq3$ by $\{\mathrm{(f)}, \mathrm{(h)}, \mathrm{(i)}\}$, thus $ch_2(v)\geq6-2-3\times\frac{1}{3}-\max\{1+2\times\frac{5}{6}, 2\times1+\frac{2}{3}, 1+\frac{5}{6}+\frac{3}{4}\}-2\times\frac{1}{6}=0$. If all vertices are $5^{+}$-vertices in $N_2(v)$, then $n_b(v)\leq4$ by $\{\mathrm{(f)}, \mathrm{(h)}, \mathrm{(h)}\}$, it follows that $ch_2(v)\geq6-2-3\times\frac{1}{3}-\max\{2\times1+\frac{5}{6}, 2\times1+\frac{3}{4}\}-\frac{1}{6}=0$ by R4 and R6.
Eventually, we consider $n_3(v)=0$. If $n_{5^{+}}(v)=0$, i.e. $d(z)=4$ for all $z\in N(v)$, then $ch_2(v)\geq6-2-3\times\frac{1}{3}-3\times\frac{2}{3}-6\times\frac{1}{6}=0$ by R3. If $n_{5^{+}}(v)=1$, it follows that $ch_2(v)\geq6-2-3\times\frac{1}{3}-2\times\frac{2}{3}-\frac{5}{6}-5\times\frac{1}{6}=0$ by R5. If $n_{5^{+}}(v)=2$, then $ch_2(v)\geq6-2-3\times\frac{1}{3}-\max\{2\times\frac{2}{3}+1, 2\times\frac{5}{6}+\frac{2}{3}\}-4\times\frac{1}{6}=0$ by R5 and R6. If $n_{5^{+}}(v)=3$, then $ch_2(v)\geq6-2-3\times\frac{1}{3}-\max\{1+\frac{5}{6}+\frac{2}{3}, 3\times\frac{5}{6}\}-3\times\frac{1}{6}=0$ by R5 and R6. If $n_{5^{+}}(v)=4$, then $ch_2(v)\geq6-2-3\times\frac{1}{3}-\max\{2\times1+\frac{2}{3}, 2\times\frac{5}{6}+1\}-2\times\frac{1}{6}=0$ by R5 and R6. If $n_{5^{+}}(v)=5$, then $ch_2(v)\geq6-2-3\times\frac{1}{3}-2\times1-\frac{5}{6}-\frac{1}{6}=0$ by R5 and R6. If $n_{5^{+}}(v)=6$, then $ch_2(v)\geq 6-2-3\times\frac{1}{3}-3\times1=0$ by R6.\medskip

\textbf{When $d(v)=5$}. Let $f_1=v_1v_2v$, $f_2=v_3v_4v$ be two $3$-faces incident with $v$. Similarly, we consider whether $v$ is bad.\medskip

$\bullet$ $v$ is bad but not dangerous in $G^{'}$;%, then $ch_2(v)\geq0$.

\noindent
Assume $f_2$ is worse, if $d(v_1)=3$, then $d(v_2)\geq5$, and note that $v_2$ is not bad by $\{\mathrm{(d)}, \mathrm{(f)}\}$, it follows from $\{\mathrm{(a)}, \mathrm{(b)}, \mathrm{(d)}\}$ that $\varpi(v)\leq2$, thus $ch_2(v)=5-2-3\times \frac{1}{3}-1-\frac{2}{3}-2\times\frac{1}{6}=0$ by R3 and R4. If $d(v_1)=d(v_2)=4$, then $\varpi(v)\leq2$ by $\{\mathrm{(a)}, \mathrm{(d)}, \mathrm{(k)}\}$, thus $ch_2(v)\geq 5-2-3\times \frac{1}{3}-1-\frac{2}{3}-2\times \frac{1}{6}=0$ by R3. If $d(v_1)=4$ and $d(v_2)\geq5$, then $v_2$ is not bad by $\{\mathrm{(d)}, \mathrm{(i)}\}$ and  $ch_2(v)\geq 5-2-3\times \frac{1}{3}-1-\frac{1}{2}-3\times \frac{1}{6}=0$ by R3 and R5. If $d(v_1)\geq5$ and $d(v_2)\geq5$, then it follows that $n_b(v)\leq1$ by $\{\mathrm{(d)}, \mathrm{(h)}\}$. If $n_b(v)=1$, w.l.o.g., let $v_1$ be a bad vertex. If $v_1$ is also dangerous, then $d(v_5)\geq4$ by Lemma \ref{bad}(1), and it follows that $ch_2(v)\geq 5-2-3\times \frac{1}{3}-1-\frac{3}{4}-\frac{1}{6}=\frac{1}{12}$ by R3 and R6. If $v_1$ is not dangerous, then $ch_2(v)\geq 5-2-3\times \frac{1}{3}-1-\frac{2}{3}-2\times\frac{1}{6}=0$. If $n_b(v)=0$, then $ch_2(v)\geq 5-2-3\times \frac{1}{3}-1-\frac{2}{3}-2\times\frac{1}{6}=0$.\medskip

$\bullet$ $v$ is neither bad nor dangerous in $G^{'}$;%, then $ch_2(v)\geq0$.

\noindent
Let $N_3(v)=\{v_1, v_2, v_3, v_4\}$, we denote $N_{3b}(v)$, $N_{3}^{*}(v)$ the set of bad vertices and $3$-vertices in $N_3(v)$ respectively. For simplicity, let $n_{3b}(v)=|N_{3b}(v)|$, $n_3^{*}(v)=|N_{3}^{*}|$. If $n_3^{*}(v)=2$, say $v_1$ and $v_3$, then $d(v_2)\geq5$, $d(v_4)\geq5$. We get $n_{3b}(v)\leq1$ by $\{\mathrm{(f)}, \mathrm{(f)}\}$. If $n_{3b}(v)=1$, then $\varpi(v)\leq1$ by Lemma \ref{4redu60}(5),
%$\{\mathrm{(a)}, \mathrm{(b)}, \mathrm{(f)}\}$, $\{\mathrm{(a)}, \mathrm{(a)}, \mathrm{(g)}\}$
 thus $ch_2(v)\geq 5-2-3\times \frac{1}{3}-1-\frac{5}{6}-\frac{1}{6}=0$. Otherwise, $n_{3b}(v)=0$, if $d(v_5)\geq4$, then $ch_2(v)\geq 5-2-3\times \frac{1}{3}-2\times\frac{5}{6}-2\times\frac{1}{6}=0$. If $d(v_5)=3$, it follows from Lemma \ref{local} that both $v_2$ and $v_4$ are well vertices, thus $ch_2(v)\geq 5-2-3\times \frac{1}{3}-2\times\frac{5}{6}-2\times\frac{1}{12}-\frac{1}{6}=0$. If $n_3^{*}(v)=1$, say $v_1$, then $d(v_2)\geq5$. When $d(v_3)=d(v_4)=4$, if $v_1$ is bad, then $\varpi(v)\leq2$ by $\{\mathrm{(a)}, \mathrm{(g)}, \mathrm{(k)}\}$, thus $ch_2(v)\geq 5-2-3\times \frac{1}{3}-1-\frac{2}{3}-2\times\frac{1}{6}=0$ by R3 and R4. Otherwise, $v_1$ is not bad, then $\varpi(v)\leq3$ by $\{\mathrm{(a)}, \mathrm{(b)}, \mathrm{(l)}\}$, thus $ch_2(v)\geq 5-2-3\times \frac{1}{3}-\frac{5}{6}-\frac{2}{3}-3\times\frac{1}{6}=0$. When $d(v_i)\geq5$ for some $i\in\{3,4\}$, if $n_{3b}(v)=2$, then $d(v_5)\geq4$ by $\{\mathrm{(a)}, \mathrm{(f)}, \mathrm{(i)}\}$. Moreover, $\varpi(v)\leq1$ by $\{\mathrm{(g)}, \mathrm{(j)}\}$, thus $ch_2(v)\geq 5-2-3\times \frac{1}{3}-1-\frac{5}{6}-\frac{1}{6}=0$. If $n_{3b}(v)=1$, first, suppose $v_i$ is bad, then $\varpi(v)\leq2$ by $\{\mathrm{(a)}, \mathrm{(b)}, \mathrm{(i)}\}$, thus $ch_2(v)\geq 5-2-3\times \frac{1}{3}-2\times\frac{5}{6}-2\times\frac{1}{6}=0$. Second, suppose $v_2$ is bad, it follows from $\{\mathrm{(a)}, \mathrm{(g)}, \mathrm{(m)}\}$ that $\varpi(v)\leq2$, thus $ch_2(v)\geq 5-2-3\times \frac{1}{3}-1-\frac{2}{3}-2\times\frac{1}{6}=0$.
If neither $v_2$ nor $v_i$ is bad, then $ch_2(v)\geq 5-2-3\times \frac{1}{3}-\frac{5}{6}-\frac{2}{3}-3\times\frac{1}{6}=0$. When $d(v_3)\geq5$, $d(v_4)\geq5$, $n_{3b}(v)\leq2$ by $\{\mathrm{(f)}, \mathrm{(h)}\}$. If $n_{3b}(v)=2$, we get $\varpi(v)\leq1$ by Lemma \ref{bad}(1), $\{\mathrm{(a)}, \mathrm{(b)}, \mathrm{(h)}\}$, thus $ch_2(v)\geq 5-2-3\times \frac{1}{3}-\max\{1+\frac{5}{6}+\frac{1}{6}, 1+\frac{3}{4}+\frac{1}{6}, 1+\frac{2}{3}+2\times\frac{1}{6}\}=0$ by R5-R6. Otherwise, $ch_2(v)\geq 5-2-3\times \frac{1}{3}-\max\{1+\frac{2}{3}, \frac{5}{6}+\frac{3}{4}\}-2\times\frac{1}{6}=0$. Next, we consider $n_3^{*}(v)=0$, which means $d(z)\geq4$ for all $z\in N_3(v)$. If $d(z)=4$ for all $z\in N_3(v)$, then $\varpi(v)\leq4$ by $\{\mathrm{(a)}, \mathrm{(l)}, \mathrm{(l)}\}$, it follows that $ch_2(v)\geq 5-2-3\times \frac{1}{3}-2\times \frac{2}{3}-4\times\frac{1}{6}=0$ by R3. If there exist a $5^{+}$ vertex in $N_3(v)$, say $v_1$, if $v_1$ is bad, then $\varpi(v)\leq3$ by $\{\mathrm{(a)}, \mathrm{(j)}, \mathrm{(l)}\}$, thus $ch_2(v)\geq 5-2-3\times \frac{1}{3}-\frac{2}{3}-\frac{5}{6}-3\times \frac{1}{6}=0$. Otherwise, $v_1$ is not bad, it follows that $ch_2(v)\geq 5-2-3\times \frac{1}{3}-2\times \frac{2}{3}-4\times \frac{1}{6}=0$ by R3 and R5. If there are two $5^{+}$-vertices in $N_3(v)$, say $v_1$ and $v_2$ or $v_1$ and $v_3$, if $n_{3b}(v)=2$, then $\varpi(v)\leq2$ by $\{\mathrm{(a)}, \mathrm{(h)}, \mathrm{(l)}\}$, $\{\mathrm{(a)}, \mathrm{(j)}, \mathrm{(j)}\}$, then $ch_2(v)=5-2-3\times \frac{1}{3}-\max\{2\times \frac{5}{6}, 1+\frac{2}{3}\}-2\times\frac{1}{6}=0$. Otherwise if $n_{3b}(v)\leq1$, then $ch_2(v)=5-2-3\times \frac{1}{3}-\max\{\frac{5}{6}+\frac{2}{3}, \frac{3}{4}+\frac{2}{3}\}-3\times \frac{1}{6}=0$. If there are three $5^{+}$-vertices in $N_3(v)$, and if $n_{3b}(v)=3$, then $\varpi(v)\leq1$ by $\{\mathrm{(a)}, \mathrm{(h)}, \mathrm{(j)}\}$, thus $ch_2(v)\geq5-2-3\times \frac{1}{3}-\frac{5}{6}-1-\frac{1}{6}=0$ by R5 and R6. Otherwise $n_{3b}(v)\leq2$, then $ch_2(v)\geq5-2-3\times \frac{1}{3}-\max\{\frac{2}{3}+1, \frac{3}{4}+\frac{5}{6}\}-2\times\frac{1}{6}=0$. If $d(z)\geq5$ for all $z\in N_3(v)$, then $n_{3b}(v)\leq3$ by $\{\mathrm{(h)}, \mathrm{(h)}\}$, thus $ch_2(v)\geq5-2-3\times \frac{1}{3}-1-\frac{3}{4}-\frac{1}{6}=\frac{1}{12}>0$.\medskip

This completes the proof of Claim \ref{np1}.
\end{proof}

Now, we are ready to verify all vertices in $G'$ satisfying $ch_2(v)\geq0$.\medskip

\textbf{Let $v$ be a $3$-vertex in $G'$}. Then $ch_2(v)=ch_1(v)=1-3\times \frac{1}{3}=0$ by R1.\medskip

\textbf{Let $v$ be a $4$-vertex in $G'$}. Then $ch_2(v)=ch_1(v)\geq2-\max\{1+3\times \frac{1}{3}, \frac{1}{2}+\frac{2}{3}+2\times \frac{1}{3}, \frac{5}{6}+3\times\frac{1}{3}\}=0$ by R2.\medskip

\textbf{Let $v$ be a $5^{+}$-vertex in $G'$}. Suppose $v$ is not dangerous, then $ch_2(v)\geq0$ by Claim \ref{np1}. Next, we consider the case $v$ is dangerous.

If $d(v)$ is even, then
$f_{3,3}(v)=\frac{d(v)}{2}-2$, thus we obtain that $ch_2(v)=d(v)-2-\frac{4}{3}(\frac{d(v)}{2}-2)-\frac{1}{3}(d(v)-f_3(v))-1-\frac{7}{6}=\frac{d(v)-9}{6} \geq0$ when $d(v)\geq9$.\medskip

\textbf{Let $v$ be a $8$-vertex in $G'$}. Note that $ch_2(v)=2-\frac{7}{6}f_{3,4}(v)-f_{4,4}(v)-f_{4b}(v)$. Let $f_3$ as well as $f_4$ be worst and $f_1$ and $f_2$ be the rest two $3$-faces. If $v$ is bad, then $\varpi(v)\leq2$ by $\{\mathrm{(c)}, \mathrm{(c)}, \mathrm{(e)}, \mathrm{(l)}\}$, and it follows that $ch_2(v)\geq8-2-4\times\frac{1}{3}-2\times\frac{4}{3}-1-\frac{2}{3}-2\times\frac{1}{6}=0$ by R3 and R5. Otherwise, $ch_2(v)\geq8-2-4\times\frac{1}{3}-2\times\frac{4}{3}-\max\{2\times\frac{2}{3}+4\times\frac{1}{6}, 2\times\frac{5}{6}+2\times \frac{1}{6}\}=0$ by R3-R5.\medskip

\textbf{Let $v$ be a $6$-vertex in $G'$}. We assume that $f_1=v_1v_2v$, $f_2=v_3v_4v$ and $f_3$ is worst.

$\bullet$ $v$ is bad and dangerous in $G^{'}$;%, then $ch_2(v)\geq0$.

\noindent
W.l.o.g, assume $f_2$ is worse. If $d(v_1)=d(v_2)=4$, then $ch_2(v)=ch_1(v)=6-2-3\times \frac{1}{3}-\frac{4}{3}-1-\frac{2}{3}=0$. If $d(v_i)\geq5$ for some $i\in\{1,2\}$, then $v_i$ can not be bad by $\{\mathrm{(c)}, \mathrm{(d)}, \mathrm{(i)}\}$, and $\varpi(v)\leq1$ by $\{\mathrm{(c)}, \mathrm{(e)}, \mathrm{(m)}\}$, thus $ch_2(v)\geq 6-2-3\times \frac{1}{3}-\frac{4}{3}-1-\frac{1}{2}-\frac{1}{6}=0$ by R4 and R5. If $d(v_i)\geq5$ for all $i\in\{1,2\}$, then $n_b(v)\leq1$ by $\{\mathrm{(c)}, \mathrm{(d)}, \mathrm{(h)}\}$, thus $ch_2(v)\geq 6-2-3\times \frac{1}{3}-\frac{4}{3}-1-\frac{1}{2}-\frac{1}{6}=0$ by R3 and R6.

$\bullet$ $v$ is not bad but dangerous in $G^{'}$;%, then $ch_2(v)\geq0$.

\textbf{Case 1.} $n_3(v)=4$;

\noindent
Then $n_{5^{+}}(v)=2$. It follows that $ch_2(v)=ch_1(v)\geq 6-2-3\times \frac{1}{3}-\frac{4}{3}-2\times\frac{5}{6}$=0.

\textbf{Case 2.} $n_3(v)=3$;

\noindent
Let $v_1$ be another $3$-vertex, then $d(v_2)\geq5$. If $d(v_3)=d(v_4)=4$, then $n_b(v)=0$ by Lemma \ref{4redu60}(2), and $\varpi(v)\leq1$ by $\{\mathrm{(a)}\,$or$\, \mathrm{(b)}, \mathrm{(c)}, \mathrm{(k)}\}$, thus $ch_2(v)\geq 6-2-3\times \frac{1}{3}-\frac{4}{3}-\frac{5}{6}-\frac{2}{3}-\frac{1}{6}=0$ by R3 and R4. If $d(v_i)\geq5$ for some $i\in\{3,4\}$, then $n_b(v)=0$ by $\{\mathrm{(a)}, \mathrm{(c)}, \mathrm{(i)}\}$ and Lemma \ref{4redu60}. Moreover, both $v_2$ and $v_i$ are well vertices by Lemma \ref{local}, thus $ch_2(v)\geq 6-2-3\times \frac{1}{3}-\frac{4}{3}-\frac{5}{6}-\frac{2}{3}-2\times\frac{1}{12}=0$ by R8. If $d(v_i)\geq5$ for all $i\in\{3,4\}$, note that $n_b(v)=1$ by $\{\mathrm{(a)}, \mathrm{(c)}, \mathrm{(h)}\}$ and $v_2$ must be well vertices by Lemma \ref{local}, then $ch_2(v)\geq 6-2-3\times \frac{1}{3}-\frac{4}{3}-\frac{5}{6}-\frac{3}{4}-\frac{1}{12}=0$ by R8.

\textbf{Case 3.} $n_3(v)=2$;

\noindent
If $n_4(v)=4$, then $\varpi(v)\leq2$ by $\{\mathrm{(c)}, \mathrm{(k)}, \mathrm{(l)}\}$, it follows that $ch_2(v)=6-2-3\times \frac{1}{3}-\frac{4}{3}-2\times \frac{2}{3}-2\times\frac{1}{6}=0$ by R3. We denote $N_4(v)=N(v)\backslash\{v_5, v_6\}$, suppose $n_{5^{+}}(v)=1$. If $n_b(v)=1$, then $\varpi(v)\leq1$ by $\{\mathrm{(c)}, \mathrm{(i)}, \mathrm{(l)}\}$, $\{\mathrm{(c)}, \mathrm{(j)}, \mathrm{(k)}\}$, it follows that $ch_2(v)\geq 6-2-3\times \frac{1}{3}-\frac{4}{3}-\frac{2}{3}-\frac{5}{6}-\frac{1}{6}=0$ by R3 and R5. Otherwise if $n_b(v)=0$, then $\varpi(v)\leq2$ by $\{\mathrm{(c)}, \mathrm{(l)}, \mathrm{(m)}\}$, it follows that $ch_2(v)\geq 6-2-3\times \frac{1}{3}-\frac{4}{3}-2\times \frac{2}{3}-2\times\frac{1}{6}=0$ by R3 and R5. If $n_{5^{+}}(v)=2$, say $v_1$, $v_2$ or $v_1$, $v_3$. In the former case, by Lemma \ref{4redu60}(4), we get that $n_b(v)\leq1$. Suppose $n_b(v)=1$, say $v_1$, if $v_1$ is also dangerous, then $\varpi(v)\leq1$ by Lemma \ref{bad}(6), then $ch_2(v)\geq 6-2-3\times \frac{1}{3}-\frac{4}{3}-\frac{2}{3}-\frac{3}{4}-\frac{1}{6}=\frac{1}{12}>0$. Otherwise, we get $ch_2(v)\geq 6-2-3\times \frac{1}{3}-\frac{4}{3}-\frac{2}{3}-\frac{2}{3}-2\times\frac{1}{6}=0$ by R6. Suppose $n_b(v)=0$, then $ch_2(v)\geq 6-2-3\times \frac{1}{3}-\frac{4}{3}-2\times\frac{2}{3}-2\times\frac{1}{6}=0$ by R3 and R6. In the latter case, suppose $n_b(v)=2$, then $ch_2(v)=ch_1(v)=6-2-3\times \frac{1}{3}-\frac{4}{3}-2\times \frac{5}{6}=0$. If $n_b(v)=1$, then $\varpi(v)\leq1$ by $\{\mathrm{(c)}, \mathrm{(j)}, \mathrm{(m)}\}$, it follows that $ch_2(v)\geq 6-2-3\times \frac{1}{3}-\frac{4}{3}-\frac{2}{3}-\frac{5}{6}-\frac{1}{6}=0$ by R5. Otherwise if $n_b(v)=0$, it follows that $ch_2(v)\geq 6-2-3\times \frac{1}{3}-\frac{4}{3}-2\times\frac{2}{3}-2\times\frac{1}{6}=0$ by R5. If $n_{5^{+}}(v)=3$, then $n_b(v)\leq2$ by $\{\mathrm{(c)}, \mathrm{(h)}, \mathrm{(i)}\}$. If $n_b(v)=2$, by the same argument, we have $ch_2(v)\geq6-2-3\times\frac{1}{3}-\frac{4}{3}-\max\{\frac{5}{6}+\frac{3}{4}, \frac{5}{6}+\frac{2}{3}+\frac{1}{6}\}=0$ by R5 and R6. If $n_b(v)\leq1$, then $ch_2(v)\geq6-2-3\times\frac{1}{3}-\frac{4}{3}-\max\{\frac{5}{6}+\frac{2}{3}, \frac{2}{3}+\frac{3}{4}\}-\frac{1}{6}=0$ by R5 and R6. If $n_{5^{+}}(v)=4$, then there are at most two bad vertices by Lemma \ref{4redu60}(4), thus $ch_2(v)\geq6-2-3\times\frac{1}{3}-\frac{4}{3}-2\times\frac{3}{4}=\frac{1}{6}>0$ by R3 and R6.

If $d(v)$ is odd, note that $v$ is also bad, it follows that $f_{3,4}(v)=f_{3b}(v)=f_{4b}(v)=f_{bb}(v)=0$. Then $ch_2(v)\geq d(v)-2-\frac{4}{3}f_{3,3}(v)-1-\frac{1}{3}(d(v)-f_3(v))=\frac{d(v)-7}{6}\geq0$ when $d(v)\geq7$.\medskip

\textbf{Let $v$ be a $5$-vertex in $G'$}. Let $f_1=v_1v_2v$ and assume that $f_2$ is worst. If there exists a $3$-vertex lying on $f_1$, say $v_1$, then $d(v_2)\geq5$ and $v_2$ is not bad by $\{\mathrm{(c)}, \mathrm{(f)}\}$. It follows that $ch_2(v)=ch_1(v)\geq5-2-3\times \frac{1}{3}-\frac{4}{3}-\frac{2}{3}=0$ by R3 and R4. If $d(v_1)=d(v_2)=4$, then $ch_2(v)=ch_1(v)=5-2-3\times \frac{1}{3}-\frac{4}{3}-\frac{2}{3}=0$ by R3. Otherwise, there exists a $5^{+}$-vertex lying on $f_1$ which is not bad by $\{\mathrm{(c)}, \mathrm{(i)}\}$, then $\varpi(v)\leq1$ by $\{\mathrm{(a)}\,$or$\, \mathrm{(b)}, \mathrm{(c)}, \mathrm{(m)}\}$. It follows that $ch_2(v)\geq 5-2-3\times \frac{1}{3}-\frac{4}{3}-\frac{1}{2}-\frac{1}{6}=0$ by R3 and R5. If there are two $5^{+}$-vertices lying on $f_1$, it follows from $\{\mathrm{(c)}, \mathrm{(h)}\}$ that $n_b(v)\leq1$, thus $ch_2(v)\geq 5-2-3\times \frac{1}{3}-\frac{4}{3}-\frac{1}{2}-\frac{1}{6}=0$ by R3 and R5.\medskip

\textbf{Let $f$ be a $6^{+}$-face in $G'$}. Then $ch_2(f)\geq 0$ by R1 and R7.\medskip

\textbf{Let $f$ be a $3$-face in $G'$}. Let $f=v_1v_2v_3$, we next consider different cases corresponding to the shape of $f$.
If $f$ is poor, it follows that $ch_2(f)\geq -2+\frac{1}{3}+2\times \frac{2}{3}+\min\{\frac{1}{3}, 2\times \frac{1}{6}\}=0$ by R1, R7 and R8. In particular, if there exists at least one $4$-vertex which is not vice on $f$, then $ch_2(f)\geq -2+1+\frac{1}{3}+\frac{2}{3}=0$ by R2.1.
If $d(v_1)=3$, $3\leq d(v_2)\leq4$, $d(v_3)\geq5$, note that $v_2$ cannot be bad, then $ch_2(f)\geq -2+\frac{1}{3}+\min\{\frac{1}{3}+\frac{4}{3}, \frac{2}{3}+1\}=0$ by R3.
If $d(v_1)=3$, $d(v_2)\geq5$, $d(v_3)\geq5$, by Lemma \ref{4redu60}, there is at most one bad vertex contained in $\{v_2,v_3\}$. It follows that $ch_2(f)\geq -2+\frac{1}{3}+\min\{1+\frac{2}{3}, 2\times \frac{5}{6}\}=0$ by R4. If $d(v_i)=4$ for each $i\in\{1,2,3\}$, then $ch_2(f)=-2+3\times \frac{2}{3}=0$ by R2.3. If $d(v_1)=4$, $d(v_2)=4$, $d(v_3)\geq5$, it follows that there is at most one bad vertex contained in $\{v_1, v_2, v_3\}$, then $ch_2(f)\geq-2+\min\{\frac{2}{3}+\frac{1}{2}+\frac{5}{6}, 3\times\frac{2}{3}\}=0$. If $d(v_1)=4$ and $d(v_i)\geq5$ for each $i\in\{2,3\}$, then $v_2$ and $v_3$ are not bad at the same time, it follows that $ch_2(v)\geq-2+\frac{2}{3}+\min\{\frac{1}{2}+\frac{5}{6}, 2\times\frac{2}{3}\}=0$. If $d(v_i)\geq5$ for all $i\in\{1,2,3\}$, then $ch_2(f)\geq-2+\min\{3\times\frac{2}{3}, \frac{1}{2}+2\times\frac{3}{4}, 2\times\frac{1}{2}+1\}=0$ by R6.

Hence, $ch_2(x)\geq0$ for all $x\in V(G')\cup F(G')$, this contradiction completes the proof of Theorem \ref{thm3}.

\end{document}